\documentclass[11pt]{article}
\usepackage{amssymb,amsfonts}
\usepackage{amsmath,amsthm,amsxtra}

\usepackage{graphicx}
\usepackage{amscd}
\usepackage{ulem}
\usepackage{makeidx}
\usepackage{fancybox}

\usepackage[linktocpage=true, colorlinks=true, linkcolor=blue, citecolor=red, urlcolor=green]{hyperref}

\newcommand{\ccc}{{\mathbf C}}

\newcommand{\nnn}{{\mathbf N}}

\newcommand{\zzz}{{\mathbf Z}}

\newtheorem{thm}{Theorem}[section]
\newtheorem{prop}{Proposition}[section]
\newtheorem{lemma}{Lemma}[section]

\newtheorem{note}{Note}[section]

\numberwithin{equation}{section}

\begin{document}

\title{Mock theta functions and indefinite modular forms III}

\author{\footnote{12-4 Karato-Rokkoudai, Kita-ku, Kobe 651-1334, 
Japan, \qquad
wakimoto.minoru.314@m.kyushu-u.ac.jp, \hspace{5mm}
wakimoto@r6.dion.ne.jp 
}{ Minoru Wakimoto}}

\date{\empty}

\maketitle

\begin{center}
Abstract
\end{center}

In this paper, we study modular transformation properties of 
a certain class of functions with indefinite quadratic forms.

\tableofcontents

\section{Introduction}

The research in this report was motivated by the observation that the 
functions 
$$
g^{[m]}_{n,m}(\tau) \, = \, 
\Big[\sum_{\substack{j, \, k \, \in \, \zzz \\[1mm]
0 \, < \, k \, \leq \, j}} 
-
\sum_{\substack{j, \, k \, \in \, \zzz \\[1mm]
j \, < \, k \, \leq \, 0}} \Big] 
(-1)^j \,
q^{(m+\frac12)(j+m-\frac12)^2 \, - \, m \, (k+m-\frac{n}{2m})^2} 
$$
can be written by the Jacobi's theta functions as 
$$
g^{[m]}_{n,m}(\tau)
\, = \,\ 
- \, \frac{(-1)^m}{2m} \cdot 
\frac{\eta(\tau)^3}{\theta_{n,m}(\tau,0)} 
\sum_{\ell \, \in \zzz/2m\zzz} 
e^{\frac{\pi in\ell}{m}} \,\ 
\dfrac{\theta_{m+\frac12, m+\frac12}^{(-)}\Big(\tau, \, -\dfrac{\ell}{m}\Big)
}{
\theta_{\frac12, \frac12}^{(-)}\Big(\tau, \, -\dfrac{\ell}{m}\Big)}
$$
when $m \in \nnn$ and $n$ is a non-negative integer satisfying $n < 2m$.
Then one will naturally ask
\begin{itemize}
\item What is the modular transformation of these functions ?
\item Do there exist $SL_2(\zzz)$-invariant families related to 
these functions ?
\end{itemize}
The aim of this paper is to give an answer to these questions.

\medskip

In this paper, we follow notations and definitions from 
\cite{W2022a}, 
\cite{W2022b}, 
\cite{W2022c}, 
\cite{W2022d}, 
\cite{W2022e},
\cite{W2022f} and 
\cite{W2023b}.

\section{Preliminaries}

Let us begin this section with recalling a very simple formula:

\begin{note} \,\ 
\label{indef3:note:2023-1120a}
Let $n$ be a positive integer and $\omega$ be an 
$n$-th primitive root of unity. Then the following formulas hold 
for a polynomial function $a(x)$ in $x$ of degree less than $n$ 
and a non-negative integer $k$ such that $k<n$. 
\begin{enumerate}
\item[{\rm 1)}] \quad $
\sum\limits_{j=0}^{n-1}\dfrac{a(\omega^{-j})}{1-\omega^jx}
\,\ = \,\ 
\dfrac{n \, a(x)}{1-x^n}$
\item[{\rm 2)}] \quad $\sum\limits_{j=0}^{n-1}
\dfrac{\omega^{-jk}}{1-\omega^j x} 
\, = \, \dfrac{n \, x^k}{1-x^n}$
\end{enumerate}
\end{note}

\begin{proof} 1) \, Since $
\dfrac{1}{1-\omega^jx}
\, = \, \dfrac{1}{1-x^n}
(1-x)(1-\omega x) \cdots 
(\widehat{1-\omega^jx}) \cdots (1-\omega^{n-1}x)$, one has
$$
\sum_{j=0}^{n-1}\frac{a(\omega^{-j})}{1-\omega^jx}
\, = \, \frac{1}{1-x^n}
\underbrace{\sum_{j=0}^{n-1}a(\omega^{-j}) \, 
(1-x)(1-\omega x) \cdots 
(\widehat{1-\omega^jx}) \cdots (1-\omega^{n-1}x)}_{\substack{ \hspace{6mm}
|| \,\ put \\[1mm] {\displaystyle f(x)
}}} 
$$
where the notation $\widehat{1-\omega^jx}$ with the hat on the top 
indicates to remove this factor $1-\omega^jx$.
Then $f$ is a polynomial of degree less than $n$ and satisfies the following: 
$$
f(\omega^{-j}) \,\ = \,\ a(\omega^{-j}) 
\underbrace{\prod_{\ell=1}^{n-1}(1-\omega^{\ell})}_{\substack{
|| \\[1mm] {\displaystyle n
}}}
= \,\ n \, a(\omega^{-j}) 
\quad {\rm for} \,\ 
j=0, \, \cdots \, , \, n-1
$$
Hence \, $f(x) \, = \, n \, a(x)$, proving 1). \, 
2) is obtained by applying 1) to $a(x) := x^k$.
\end{proof}

\medskip

We consider the mock theta function $\Phi^{(\pm)[m,s]}_1$ defined by 
\begin{equation}
\Phi^{(\pm)[m,s]}_1(\tau, z_1,z_2,0) \,\ := \,\ 
\sum_{j \in \zzz} \, (\pm 1)^j \, 
\frac{e^{2\pi imj(z_1+z_2)+2\pi isz_1} \, q^{mj^2+sj}
}{1-e^{2\pi iz_1}q^j}
\label{indef:eqn:2023-1120b}
\end{equation}
for $m \in \frac12 \nnn$ and $s \in \frac12 \zzz$. 
We write simply $\Phi^{[m,s]}_1$ for $\Phi^{(+)[m,s]}_1$.

\medskip

\begin{lemma} 
\label{indef3:lemma:2023-1122a}
Let $m \in \frac12 \nnn$ and $n \in \nnn$ and $p \in \zzz_{\geq 0}$
such that $p<n$. Then 
\begin{enumerate}
\item[{\rm 1)}] \,\ ${\displaystyle 
\sum_{k=0}^{n-1} e^{\frac{2\pi ipk}{n}}
\Phi^{[nm,0]}_1\Big(\tau, \, z_1 - \frac{k}{n}, \, 
z_2 \pm \frac{k}{n}, \, 0\Big)
= \, 
n \, \sum_{j \in \zzz} \, 
\frac{e^{2\pi inmj(z_1+z_2)+2\pi ipz_1}q^{nmj^2+pj}
}{1-e^{2\pi inz_1} \, q^{nj}}
}$
\item[{\rm 2)}] \,\ ${\displaystyle
\sum_{k=0}^{n-1} e^{-\frac{2\pi ipk}{n}}
\Phi^{[nm,0]}_1\Big(\tau, \, z_1 + \frac{k}{n}, \, 
z_2 \pm \frac{k}{n}, \, 0\Big)
= \, 
n \, \sum_{j \in \zzz} \, 
\dfrac{e^{2\pi inmj(z_1+z_2)+2\pi ipz_1}q^{nmj^2+pj}
}{1-e^{2\pi inz_1} \, q^{nj}}
}$
\end{enumerate}
\end{lemma}

\begin{proof} 1) \,\ First we consider \, $z_2+\frac{k}{n}$. 
Then, by \eqref{indef:eqn:2023-1120b}, we have
\begin{subequations}
{\allowdisplaybreaks
\begin{eqnarray}
& & \hspace{-10mm}
\sum_{k=0}^{n-1} e^{\frac{2\pi ipk}{n}}
\Phi^{[nm,0]}_1\Big(\tau, \, z_1-\frac{k}{n}, \, z_2+\frac{k}{n}, \, 0\Big)
\,\ = \,\ 
\sum_{k=0}^{n-1} e^{\frac{2\pi ipk}{n}}
\sum_{j \in \zzz}
\frac{e^{2\pi inmj(z_1+z_2)}q^{nmj^2}
}{1-e^{2\pi i(z_1-\frac{k}{n})} \, q^j}
\label{indef:eqn:2023-1120a1}
\\[2mm]
& & \hspace{30mm}
= \,\ 
\sum_{j \in \zzz}e^{2\pi inmj(z_1+z_2)}q^{nmj^2} 
\sum_{k=0}^{n-1} \frac{(e^{\frac{2\pi ik}{n})^p}
}{1-e^{-\frac{2\pi ik}{n}} \, e^{2\pi iz_1}q^j}
\nonumber
\end{eqnarray}}
which is, by Note \ref{indef3:note:2023-1120a}, equal to 
\begin{equation}
= 
\sum_{j \in \zzz}e^{2\pi inmj(z_1+z_2)}q^{nmj^2} 
\frac{n  (e^{2\pi iz_1}q^j)^p }{1-(e^{2\pi iz_1}q^j)^{n}} 
= \, 
n  \sum_{j \in \zzz} \, 
\frac{e^{2\pi inmj(z_1+z_2)+2\pi ipz_1}q^{nmj^2+pj}
}{1-e^{2\pi inz_1} \, q^{nj}}
\label{indef:eqn:2023-1120a2}
\end{equation}
\end{subequations}
proving 1) for $z_2+\frac{k}{n}$. \,\ 
Next we consider the case  \, $z_2-\frac{k}{n}$. 
{\allowdisplaybreaks
\begin{eqnarray*}
& & \hspace{-10mm}
\sum_{k=0}^{n-1} e^{\frac{2\pi ipk}{n}}
\Phi^{[nm,0]}_1\Big(\tau, \, z_1-\frac{k}{n}, \, z_2-\frac{k}{n}, \, 0\Big)
\,\ = \,\ 
\sum_{k=0}^{n-1} e^{\frac{2\pi ipk}{n}}
\sum_{j \in \zzz}
\frac{e^{2\pi inmj(z_1+z_2-\frac{2k}{n})}q^{nmj^2}
}{1-e^{2\pi i(z_1-\frac{k}{n})} \, q^j}
\\[0mm]
&=&
\sum_{k=0}^{n-1} e^{\frac{2\pi ipk}{n}}
\sum_{j \in \zzz}
\frac{e^{2\pi inmj(z_1+z_2)}
\overbrace{e^{-4\pi ikmj}}^{1}
q^{nmj^2}
}{1-e^{2\pi i(z_1-\frac{k}{n})} \, q^j}
=
\sum_{k=0}^{n-1} e^{\frac{2\pi ipk}{n}}
\sum_{j \in \zzz}
\frac{e^{2\pi inmj(z_1+z_2)}q^{nmj^2}
}{1-e^{2\pi i(z_1-\frac{k}{n})} \, q^j}
\end{eqnarray*}}
which is the same with \eqref{indef:eqn:2023-1120a1}, so 
is equal to \eqref{indef:eqn:2023-1120a2}.

\medskip

The proof for the claim 2) is quite similar.
\end{proof}

\medskip

Changing the notation $(m,n) \rightarrow (\frac12, 2m)$ in the 
above Lemma \ref{indef3:lemma:2023-1122a} gives the following:

\medskip

\begin{lemma} 
\label{indef3:lemma:2023-1122b}
\label{note:2023-1111a}
For $m \in \frac12 \nnn$ and $p \in \zzz$ such that 
$0 \leq p < 2m$, the following formulas hold:
\begin{enumerate}
\item[{\rm 1)}] \,\ ${\displaystyle 
\sum_{k=0}^{2m-1} 
e^{\frac{\pi ipk}{m}}
\Phi^{[m,0]}_1\Big(\tau, z_1-\frac{k}{2m}, 
z_2 \pm \frac{k}{2m}, 0\Big)
= 
2m \sum_{j \in \zzz} 
\frac{e^{2\pi imj(z_1+z_2)+2\pi ipz_1}q^{mj^2+pj}
}{1-e^{4\pi imz_1} \, q^{2mj}}
}$
\item[{\rm 2)}] \,\ ${\displaystyle 
\sum_{k=0}^{2m-1} 
e^{-\frac{\pi ipk}{m}}
\Phi^{[m,0]}_1\Big(\tau, z_1+\frac{k}{2m}, 
z_2 \pm \frac{k}{2m}, 0\Big)
\, = \, 
2m \sum_{j \in \zzz} 
\frac{e^{2\pi imj(z_1+z_2)+2\pi ipz_1}q^{mj^2+pj}
}{1-e^{4\pi imz_1} \, q^{2mj}}
}$
\end{enumerate}
\end{lemma}

\medskip

The folowing Lemma \ref{indef3:lemma:2023-1122c} is obtained 
by similar calculation with the proof of Lemma 2.1 in \cite{W2022b}.

\medskip

\begin{lemma} \quad 
\label{indef3:lemma:2023-1122c}
Let $m \in \frac12 \nnn$, $s \in \frac12 \zzz$ and $n \in \nnn$. Then
\begin{enumerate}
\item[{\rm 1)}] \,\ $\Phi^{(\pm)[m,s+n]}_1(\tau, z_1, z_2,0) \,\ = \,\ 
\Phi^{(\pm)[m,s]}_1(\tau, z_1, z_2,0)$
$$
- \,\ \sum_{j=0}^{n-1} \, 
e^{\pi i(s+j)(z_1-z_2)} \, 
q^{-\frac{1}{4m}(s+j)^2} \, 
\theta_{s+j, m}^{(\pm)}(\tau, \, z_1+z_2)
$$
\item[{\rm 2)}] \,\ $\Phi^{(\pm)[m,s-n]}_1(\tau, z_1, z_2,0) \,\ = \,\ 
\Phi^{(\pm)[m,s]}_1(\tau, z_1, z_2,0)$
$$
+ \,\ \sum_{j=1}^n \, 
e^{\pi i(s-j)(z_1-z_2)} \, 
q^{-\frac{1}{4m}(s-j)^2} \, 
\theta_{s-j, m}^{(\pm)}(\tau, \, z_1+z_2)
$$
\end{enumerate}
\end{lemma}

\medskip

In this paper, the Kac-Peterson's identity stated in Lemma 2.1 of 
\cite{W2022e} will be used in the following form:

\medskip

\begin{lemma} \quad 
\label{indef3:lemma:2023-1122d}
\label{note:2023-1022a} 
For $s \in \frac12 \zzz_{\rm odd}$ and $k \in \zzz$, the 
following formula holds:
\begin{equation}
\Phi^{(-)[\frac12, s]}_1(\tau, z, -z+2k\tau,0) 
\, = \, 
- \, i \, e^{-2\pi ikz} \, 
\dfrac{\eta(\tau)^3}{\vartheta_{11}(\tau,z)}
\label{indef3:eqn:2023-1122a}
\end{equation}
\end{lemma}

\begin{proof} Letting $m=s=\frac12$ and $(z_1, z_2)=(z, -z+2k\tau)$
in Lemma \ref{indef3:lemma:2023-1122c}, we have 
{\allowdisplaybreaks
\begin{eqnarray*}
& & \hspace{-20mm}
\Phi^{(-)[\frac12,\frac12+n]}_1(\tau, z, -z+2k\tau,0) 
\,\ = \,\ 
\Phi^{(-)[\frac12,\frac12]}_1(\tau, z, -z+2k\tau,0)
\\[2mm]
& & 
- \,\ \underbrace{\sum_{j=0}^{n-1}
e^{\pi i(\frac12+j)(2z-2k\tau)} \, 
q^{-\frac{1}{2}(\frac12+j)^2} \, 
\theta_{\frac12+j, \frac12}^{(-)}(\tau, \, 2k\tau)}_{A}
\\[3mm]
& & \hspace{-20mm}
\Phi^{(-)[\frac12,\frac12-n]}_1(\tau, z, -z+2k\tau,0) 
\,\ = \,\ 
\Phi^{(-)[\frac12,\frac12]}_1(\tau, z, -z+2k\tau,0)
\\[2mm]
& & 
+ \,\ \underbrace{\sum_{j=1}^n
e^{\pi i(\frac12-j)(2z-2k\tau)} \, 
q^{-\frac{1}{2}(\frac12-j)^2} \, 
\theta_{\frac12-j, \frac12}^{(-)}(\tau, \, 2k\tau)}_{B}
\end{eqnarray*}}
for $n \in \nnn$. \, Since 
$$
\theta_{\frac12 \pm j, \, \frac12}^{(-)}(\tau, 2k\tau)
\,\ = \,\ 
(-1)^j \, \theta_{\frac12, \frac12}^{(-)}(\tau, 2k\tau)
\,\ = \,\ 
- \, i \, (-1)^j \, \vartheta_{11}(\tau, k \tau) 
\,\ = \,\ 0, 
$$
we have \,\ $A=B=0$, so the above equations become as follows:
$$
\Phi^{(-)[\frac12, \frac12 \pm n]}_1(\tau, z, -z+2k\tau,0) 
\, = \, 
\Phi^{(-)[\frac12,\frac12]}_1(\tau, z, -z+2k\tau,0)
\,\ = \,\ 
- \, i \, e^{-2\pi ikz} \, 
\frac{\eta(\tau)^3}{\vartheta_{11}(\tau,z)}
$$
by Lemma 2.1 in \cite{W2022e}, proving \eqref{indef3:eqn:2023-1122a}.
\end{proof}

\section{Functions $g^{[m]}_{n, \nu}(\tau)$, $h^{[m]}_{n, \nu}(\tau,z)$
and $G^{\, [m]}_{n, \nu}(\tau,z)$}




\begin{lemma} 
\label{indef3:lemma:2023-1023a}
For $m \in \frac12 \nnn$ and $s \in \frac12 \zzz_{\rm odd}$, 
the following formulas hold:
\begin{subequations}
\begin{enumerate}
\item[{\rm 1)}] \quad $\theta^{(-)}_{s, m+\frac12}
\Big(\tau, \, - \, \dfrac{2m(z_1-z_2)}{2m+1}\Big) \, 
\Phi^{[m,0]}_1(\tau, \, z_1, \, z_2, \, 0)$
{\allowdisplaybreaks
\begin{eqnarray}
& & 
- \,\ e^{-\frac{\pi ims}{m+\frac12}(z_1-z_2)} \, 
\Big[\sum_{\substack{j, \, k \, \in \, \zzz \\[1mm]
0 \, \leq \, k \, < \, 2mj}}
- 
\sum_{\substack{j, \, k \, \in \, \zzz \\[1mm]
2mj \, \leq \, k \, < \, 0}}\Big] \, f_{j,k} \, 
\theta_{k,m}(\tau, z_1+z_2)
\nonumber
\\[2mm]
&= &
- \,\ i \,\ 
\theta_{s, \, m+\frac12}^{(-)}
\Big(\tau, \, z_1+z_2+\frac{z_1-z_2}{2m+1}\Big) \cdot 
\frac{\eta(\tau)^3}{\vartheta_{11}(\tau,z_1)}
\label{indef3:eqn:2023-1023a1}
\end{eqnarray}}
\item[{\rm 2)}] \quad $\theta^{(-)}_{s, m+\frac12}
\Big(\tau, \, - \, \dfrac{2m(z_1-z_2)}{2m+1}\Big) \, 
\Phi^{[m,0]}_1(\tau, \, z_1, \, z_2, \, 0)$
{\allowdisplaybreaks
\begin{eqnarray}
& & 
+ \,\ 
e^{-\frac{\pi ims}{m+\frac12}(z_1-z_2)} \hspace{-3mm}
\sum_{\substack{r \, \in \, \zzz \\[1mm] 0 \,\leq \, r \, < \, 2m}} 
\Big[\sum_{\substack{j, \, p \, \in \, \zzz \\[1mm]
0 \, < \, p \, \leq \, j}} 
-
\sum_{\substack{j, \, p \, \in \, \zzz \\[1mm]
j \, < \, p \, \leq \, 0}} \Big] \, 
f_{-j, \, r-2pm} \, \theta_{r,m}(\tau,z_1+z_2)
\nonumber
\\[2mm]
&= &
- \,\ i \,\ 
\theta_{s, \, m+\frac12}^{(-)}
\Big(\tau, \, z_1+z_2+\frac{z_1-z_2}{2m+1}\Big) \cdot 
\frac{\eta(\tau)^3}{\vartheta_{11}(\tau,z_1)}
\label{indef3:eqn:2023-1023a2}
\end{eqnarray}}
\end{enumerate}
\end{subequations}
{\rm where}
\begin{equation}
f_{j,k} \,\ := \,\ (-1)^j \, 
q^{(m+\frac12)(j+\frac{s}{2m+1})^2 \, - \, \frac{1}{4m}k^2} \, 
e^{-2\pi ijm(z_1-z_2)\, + \, \pi ik(z_1-z_2)}
\label{indef3:eqn:2023-1023a3}
\end{equation}
\end{lemma}

\medskip

\begin{proof} 1) \,\ We put 
\begin{equation}
A^{[m,s]}(\tau, z_1,z_2) \,\ := \,\ 
\sum_{j, \, k \, \in \, \zzz} (-1)^j \, 
\frac{e^{2\pi ikm(z_1+z_2)} \, q^{\frac12 j^2+mk^2-sj}
}{1-e^{2\pi iz_1} \, q^{j+k}}
\label{indef3:eqn:2023-1023b}
\end{equation}
and compute the RHS of this function as follows.

\medskip

Firstly, by putting $k=r-j$, the RHS of 
\eqref{indef3:eqn:2023-1023b} becomes as follows:
\begin{subequations}
{\allowdisplaybreaks
\begin{eqnarray}
& & \hspace{-13mm}
A^{[m,s]}(\tau, z_1,z_2)
\,\ = \,\ 
\sum_{j, \, r \, \in \, \zzz} (-1)^j
\frac{e^{2\pi im(r-j)(z_1+z_2)} \, 
q^{\frac12 j^2+ m(r-j)^2 - sj}
}{1-e^{2\pi iz_1} \, q^r}
\nonumber
\\[-1mm]
&=&
\sum_{j \in \zzz} (-1)^j \, 
q^{(m+\frac12)j^2-sj} \, 
e^{-2\pi ijm(z_1+z_2)} 
\underbrace{\sum_{r \in \zzz} \frac{
e^{2\pi imr(z_1+z_2-2j\tau)}
q^{mr^2}}{1-e^{2\pi iz_1} \, q^r}}_{\substack{|| \\[-1.5mm] 
{\displaystyle \Phi^{[m,0]}_1(\tau, \, z_1, \, z_2-2j\tau, \, 0)
}}}
\nonumber
\\[2mm]
& \hspace{-3mm}
\underset{\substack{\\[0.5mm] \uparrow \\[1mm] j \, \rightarrow \, -j
}}{=} \hspace{-3mm} &
\sum_{j \in \zzz} (-1)^j \, 
q^{(m+\frac12)j^2+sj} \, 
e^{2\pi ijm(z_1+z_2)} \, 
\Phi^{[m,0]}_1(\tau, \, z_1, \, z_2+2j\tau, \, 0)
\label{indef3:eqn:2023-1023b1}
\end{eqnarray}}

Secondly, by putting $j=r-k$, the RHS of 
\eqref{indef3:eqn:2023-1023b} becomes as follows:
{\allowdisplaybreaks
\begin{eqnarray*}
& & \hspace{-13mm}
A^{[m,s]}(\tau, z_1,z_2)
\,\ = \,\ 
\sum_{k, \, r \, \in \, \zzz} (-1)^{k+r} \, 
\frac{e^{2\pi ikm(z_1+z_2)} \, 
q^{\frac12 (r-k)^2+mk^2-s(r-k)}
}{1-e^{2\pi iz_1} \, q^r}
\\[2mm]
&=&
e^{2\pi isz_1}
\sum_{k \in \zzz} (-1)^k
q^{(m+\frac12)k^2+sk} \, 
e^{2\pi ikm(z_1+z_2)} \, 
\underbrace{\sum_{r \in \zzz} (-1)^r \, 
\frac{q^{-kr}e^{-2\pi isz_1} \, q^{\frac12 r^2-sr}
}{1-e^{2\pi iz_1} \, q^r} }_{\substack{|| \\[-1.5mm] {\displaystyle 
\Phi^{(-)[\frac12,-s]}_1(\tau, \, z_1, \, -z_1-2k\tau, \, 0)
}}}
\end{eqnarray*}}
which, by Lemma \ref{indef3:lemma:2023-1122d}, is equal to
\begin{equation}
= \,\ - \,\ i \, 
\underbrace{e^{2\pi isz_1}\sum_{k \in \zzz} (-1)^k
q^{(m+\frac12)k^2+sk} \, 
e^{2\pi ikm(z_1+z_2)} \, e^{2\pi ikz_1}}_{(A)} \, 
\frac{\eta(\tau)^3}{\vartheta_{11}(\tau,z_1)}
\label{indef3:eqn:2023-1023b2}
\end{equation}
Here, by easy calculation, $(A)$ is rewritten as follows:
{\allowdisplaybreaks
\begin{eqnarray}
& & \hspace{-8mm}
(A) \, = \, 
e^{\frac{\pi ims}{m+\frac12}(z_1-z_2)} \, 
q^{- \frac{s^2}{4(m+\frac12)}}
\sum_{k \in \zzz}(-1)^k \, 
q^{(m+\frac12) (k+\frac{s}{2(m+\frac12)})^2} \, 
e^{2\pi i(m+\frac12)(k+\frac{s}{2(m+\frac12)})
(z_1+z_2+\frac{z_1-z_2}{2m+1})}
\nonumber
\\[1mm]
&=&
e^{\frac{\pi ims}{m+\frac12}(z_1-z_2)} \, 
q^{- \frac{s^2}{4(m+\frac12)}} \, 
\theta_{s, \, m+\frac12}^{(-)}
\Big(\tau, \, z_1+z_2+\frac{z_1-z_2}{2m+1}\Big)
\label{indef3:eqn:2023-1023b3}
\end{eqnarray}}
\end{subequations}
Then, by \eqref{indef3:eqn:2023-1023b1},
\eqref{indef3:eqn:2023-1023b2} and \eqref{indef3:eqn:2023-1023b3},
we obtain the following:
\begin{subequations}
{\allowdisplaybreaks
\begin{eqnarray}
& & 
\sum_{j \in \zzz} (-1)^j \, 
q^{(m+\frac12)j^2+sj} \, 
e^{2\pi ijm(z_1+z_2)} \, 
\Phi^{[m,0]}_1(\tau, \, z_1, \, z_2+2j\tau, \, 0)
\nonumber
\\[1mm]
& & \hspace{-12mm}
= \,\ - \,\ i \,\ 
e^{\frac{\pi ims}{m+\frac12}(z_1-z_2)} \, 
q^{- \frac{s^2}{4(m+\frac12)}} \, 
\theta_{s, \, m+\frac12}^{(-)}
\Big(\tau, \, z_1+z_2+\frac{z_1-z_2}{2m+1}\Big) \cdot 
\frac{\eta(\tau)^3}{\vartheta_{11}(\tau,z_1)}
\label{indef3:eqn:2023-1023c1}
\end{eqnarray}}
By Lemma 2.5 in \cite{W2022b}, the LHS of this equation is 
rewitten as follows: 
{\allowdisplaybreaks
\begin{eqnarray}
& & \hspace{-10mm}
\text{LHS of \eqref{indef3:eqn:2023-1023c1}} 
\, = \,\ 
q^{-\frac{s^2}{4(m+\frac12)}} \, 
e^{\frac{\pi ims}{m+\frac12}(z_1-z_2)} \, 
\theta^{(-)}_{s, m+\frac12}
\Big(\tau, \, - \frac{2m(z_1-z_2)}{2m+1}\Big) \, 
\Phi^{[m,0]}_1(\tau, \, z_1, \, z_2, \, 0)
\nonumber
\\[3mm]
& & 
- \,\ q^{-\frac{s^2}{4(m+\frac12)}}
\Big[\sum_{\substack{j, \, k \, \in \, \zzz \\[1mm]
0 \, \leq \, k \, < \, 2mj}}
- 
\sum_{\substack{j, \, k \, \in \, \zzz \\[1mm]
2mj \, \leq \, k \, < \, 0}}\Big]
(-1)^j \, q^{(m+\frac12)(j+\frac{s}{2(m+\frac12)})^2} \, 
q^{-\frac{1}{4m}k^2} 
\nonumber
\\[2mm]
& & \hspace{15mm}
\times \,\ 
e^{-2\pi ijm(z_1-z_2)} \, e^{\pi ik(z_1-z_2)} \, 
\theta_{k,m}(\tau, z_1+z_2)
\label{indef3:eqn:2023-1023c2}
\end{eqnarray}}
\end{subequations}
Then, by \eqref{indef3:eqn:2023-1023c1} and 
\eqref{indef3:eqn:2023-1023c2}, we obtain 
{\allowdisplaybreaks
\begin{eqnarray*}
& & \hspace{4mm}
q^{-\frac{s^2}{4(m+\frac12)}} \, 
e^{\frac{\pi ims}{m+\frac12}(z_1-z_2)} \, 
\theta^{(-)}_{s, m+\frac12}
\Big(\tau, \, - \frac{2m(z_1-z_2)}{2m+1}\Big) \, 
\Phi^{[m,0]}_1(\tau, \, z_1, \, z_2, \, 0)
\\[3mm]
& & 
- \,\ q^{-\frac{s^2}{4(m+\frac12)}}
\Big[\sum_{\substack{j, \, k \, \in \, \zzz \\[1mm]
0 \, \leq \, k \, < \, 2mj}}
- 
\sum_{\substack{j, \, k \, \in \, \zzz \\[1mm]
2mj \, \leq \, k \, < \, 0}}\Big]
(-1)^j \, q^{(m+\frac12)(j+\frac{s}{2(m+\frac12)})^2} \, 
q^{-\frac{1}{4m}k^2} 
\\[2mm]
& & \hspace{15mm}
\times \,\ 
e^{-2\pi ijm(z_1-z_2)} \, e^{\pi ik(z_1-z_2)} \, 
\theta_{k,m}(\tau, z_1+z_2)
\\[1mm]
&= &
- \, i \, 
e^{\frac{\pi ims}{m+\frac12}(z_1-z_2)} \, 
q^{- \frac{s^2}{4(m+\frac12)}} \, 
\theta_{s, \, m+\frac12}^{(-)}
\Big(\tau, \, z_1+z_2+\frac{z_1-z_2}{2m+1}\Big) \cdot 
\frac{\eta(\tau)^3}{\vartheta_{11}(\tau,z_1)}
\end{eqnarray*}}
namely
{\allowdisplaybreaks
\begin{eqnarray*}
& & \hspace{4mm}
\theta^{(-)}_{s, m+\frac12}
\Big(\tau, \, - \frac{2m(z_1-z_2)}{2m+1}\Big) \, 
\Phi^{[m,0]}_1(\tau, \, z_1, \, z_2, \, 0)
\\[3mm]
& & 
- \,\ e^{-\frac{\pi ims}{m+\frac12}(z_1-z_2)} \, 
\Big[\sum_{\substack{j, \, k \, \in \, \zzz \\[1mm]
0 \, \leq \, k \, < \, 2mj}}
- 
\sum_{\substack{j, \, k \, \in \, \zzz \\[1mm]
2mj \, \leq \, k \, < \, 0}}\Big]
(-1)^j \, q^{(m+\frac12)(j+\frac{s}{2(m+\frac12)})^2} \, 
q^{-\frac{1}{4m}k^2} 
\\[2mm]
& & \hspace{15mm}
\times \,\ 
e^{-2\pi ijm(z_1-z_2)} \, e^{\pi ik(z_1-z_2)} \, \theta_{k,m}(\tau, z_1+z_2)
\\[2mm]
&= &
- \,\ i \,\ 
\theta_{s, \, m+\frac12}^{(-)}
\Big(\tau, \, z_1+z_2+\frac{z_1-z_2}{2m+1}\Big) \cdot 
\frac{\eta(\tau)^3}{\vartheta_{11}(\tau,z_1)}
\end{eqnarray*}}

\noindent
proving 1). The claim 2) is obtained from 1) by rewriting the 2nd term 
in the LHS of \eqref{indef3:eqn:2023-1023a1} by using an easy formula 
{\allowdisplaybreaks
\begin{eqnarray}
& &
\Big[ 
\sum\limits_{\substack{j, \, k \, \in \, \zzz \\[1mm]
0 \, \leq \, k \, < \, 2mj}}
-
\sum\limits_{\substack{j, \, k \, \in \, \zzz \\[1mm]
2mj \, \leq \, k \, < \, 0}} \Big] \, 
f_{j,k} \cdot \theta_{k,m}^{(\pm)}(\tau,z)
\nonumber
\\[2mm]
&=&
- \sum_{\substack{r \, \in \, \zzz \\[1mm] 0 \,\leq \, r \, < \, 2m}} 
\Big[\sum_{\substack{j, \, p \, \in \, \zzz \\[1mm]
0 \, < \, p \, \leq \, j}} 
-
\sum_{\substack{j, \, p \, \in \, \zzz \\[0mm]
j \, < \, p \, \leq \, 0}} \Big] \, 
(\pm 1)^p \, f_{-j, \, r-2pm} \, \theta_{r,m}^{(\pm)}(\tau,z)
\label{indef3:eqn:2023-1023d}
\end{eqnarray}}

\vspace{-10mm}
\end{proof}

\medskip

Letting 
\begin{equation}
(z_1,z_2) \,\ = \,\ (z+a\tau+c, \, z+b\tau+d)
\label{indef3:eqn:2023-1203c1}
\end{equation}
in \eqref{indef3:eqn:2023-1023a2}, we obtain the following :

\medskip

\begin{lemma} 
\label{indef3:lemma:2023-1123b}
For $m \in \frac12 \nnn$ and $s \in \frac12 \zzz_{\rm odd}$ and 
$a,b,c,d \in \ccc$, the following formula holds:
{\allowdisplaybreaks
\begin{eqnarray}
& &
\theta^{(-)}_{s, m+\frac12}\Big(\tau, \,\ 
-\frac{m(a-b)}{m+\frac12} \, \tau
\, - \, 
\frac{\frac12(c-d)}{m+\frac12}\Big) \, 
\Phi^{[m,0]}_1(\tau, \, z+a\tau+c, \, z+b\tau+d, \, 0)
\nonumber
\\[3mm]
& & 
+ \,\ 
e^{-\frac{\pi ims}{m+\frac12}[(a-b)\tau+c-d]} \hspace{-3mm}
\sum_{\substack{r \, \in \, \zzz \\[1mm] 0 \,\leq \, r \, < \, 2m}} 
\Big[\sum_{\substack{j, \, p \, \in \, \zzz \\[1mm]
0 \, < \, p \, \leq \, j}} 
-
\sum_{\substack{j, \, p \, \in \, \zzz \\[1mm]
j \, < \, p \, \leq \, 0}} \Big] 
f_{-j, \, r-2pm} \, \theta_{r,m}(\tau, \, 2z+(a+b)\tau+c+d)
\nonumber
\\[2mm]
&= &
- \,\ i \,\ 
\theta_{s, \, m+\frac12}^{(-)}
\Big(\tau, \, 2z+(a+b)\tau+c+d+\frac{(a-b)\tau+c-d}{2m+1}\Big) 
\nonumber
\\[2mm]
& & \hspace{15mm}
\times \,\ 
\frac{\eta(\tau)^3}{\vartheta_{11}\Big(\tau, \, 
\dfrac{2z+(a+b)\tau+c+d}{2} + \dfrac{(a-b)\tau+c-d}{2}\Big)}
\label{indef3:eqn:2023-1023e}
\end{eqnarray}}
\end{lemma}

\medskip

\begin{lemma} 
\label{indef3:lemma:2023-1123c}
Let $m \in \frac12 \nnn$ and $\nu \in \zzz$ such that $0 \leq \nu \leq 2m$. 
Then
{\allowdisplaybreaks
\begin{eqnarray}
& & \hspace{-10mm}
(-1)^{\nu} \, q^{-m\nu^2} \, 
\theta^{(-)}_{\nu+\frac12, m+\frac12}(\tau,0) \, 
\Phi^{[m,0]}_1\Big(\tau, \,\ 
\frac{z}{2}+\nu\tau, \,\ \frac{z}{2}-\nu\tau, \,\ 0\Big)
\nonumber
\\[2mm]
& & \hspace{-8mm}
+ \hspace{-3mm}
\sum_{\substack{r \, \in \, \zzz \\[1mm] 0 \,\leq \, r \, < \, 2m}} 
\Big[\sum_{\substack{j, \, p \, \in \, \zzz \\[1mm]
0 \, < \, p \, \leq \, j}} 
-
\sum_{\substack{j, \, p \, \in \, \zzz \\[1mm]
j \, < \, p \, \leq \, 0}} \Big] 
(-1)^j \,
q^{(m+\frac12)(j+\nu-\frac{\nu+\frac12}{2(m+\frac12)})^2
\, - \, m \, (p+\nu-\frac{r}{2m})^2} 
\theta_{r,m}(\tau, \, z)
\nonumber
\\[2mm]
& & \hspace{-10mm}
= \,\ 
- \,\ (-1)^{\nu} \, \eta(\tau)^3 \, 
\frac{\theta_{\nu+\frac12, m+\frac12}^{(-)}(\tau, \, z)}{
\theta_{\frac12, \frac12}^{(-)}(\tau, z)}
\label{indef3:eqn:2023-1123a}
\end{eqnarray}}
\end{lemma}

\begin{proof} Letting $s=\frac12$ and $(a,b,c,d)=(\nu,-\nu,0,0)$ in 
\eqref{indef3:eqn:2023-1203c1}, one has
\begin{subequations}
{\allowdisplaybreaks
\begin{eqnarray}
& & \hspace{5mm}
\underbrace{\theta^{(-)}_{\frac12, m+\frac12}
\Big(\tau, \, - \, \dfrac{2m\nu\tau}{m+\frac12}\Big)}_{(A)} \, 
\Phi^{[m,0]}_1(\tau, \, z+\nu\tau, \, z-\nu\tau, \, 0)
\nonumber
\\[3mm]
& & 
+ \,\ 
q^{-\frac{m \nu}{2(m+\frac12)}} \,\ 
\sum_{\substack{r \, \in \, \zzz \\[1mm] 0 \,\leqq \, r \, < \, 2m}} 
\Big[\sum_{\substack{j, \, p \, \in \, \zzz \\[1mm]
0 \, < \, p \, \leq \, j}} 
-
\sum_{\substack{j, \, p \, \in \, \zzz \\[1mm]
j \, < \, p \, \leq \, 0}} \Big] 
f_{-j, \, r-2pm} \, \theta_{r,m}(\tau, \, 2z)
\nonumber
\\[2mm]
&= &
- \,\ i \,\ \eta(\tau)^3 \,\ 
\underbrace{\frac{\theta_{\frac12, \, m+\frac12}^{(-)}
\Big(\tau, \, 2z+\dfrac{\nu\tau}{m+\frac12}\Big) \, 
}{
\vartheta_{11}(\tau, \, z + \nu\tau)}}_{(B)}
\label{eindef3:qn:2023-1123b1}
\end{eqnarray}}

Noticing that \, $\left\{
\begin{array}{lcl}
\theta^{(-)}_{\frac12-2m\nu, m+\frac12}
&=& 
\theta^{(-)}_{\frac12+\nu-2\nu(m+\frac12), m+\frac12}
\,\ = \,\ 
(-1)^{\nu} \, \theta^{(-)}_{\nu+\frac12, m+\frac12}
\\[4mm]
\vartheta_{11}(\tau, \, z+\nu\tau)
&=&
(-1)^{\nu} \, q^{-\frac12 \nu^2} \, 
e^{-2\pi i\nu z} \, \vartheta_{11}(\tau,z)
\end{array}\right. $

\noindent
and using Lemma 1.1 and Note 1.1 in \cite{W2022c}, $(A)$ and 
$(B)$ are computed as follows:
{\allowdisplaybreaks
\begin{eqnarray*}
& & \hspace{-12mm}
(A) \,\ = \,\ 
\theta^{(-)}_{\frac12, m+\frac12}
\Big(\tau, \, - \, \dfrac{2m\nu \tau}{m+\frac12}\Big)
\,\ = \,\ 
q^{-\frac14(m+\frac12)(\frac{2m\nu}{m+\frac12})^2}
\theta^{(-)}_{\frac12-2m\nu, m+\frac12}(\tau,0)
\\[2mm]
&=&
(-1)^{\nu} \, q^{-\frac{m^2\nu^2}{m+\frac12}} \, 
\theta^{(-)}_{\nu+\frac12, m+\frac12}(\tau,0)
\\[3mm]
& & \hspace{-12mm}
(B) = \, \frac{
\theta_{\frac12, \, m+\frac12}^{(-)}
\Big(\tau, \, 2z+\dfrac{\nu\tau}{m+\frac12}\Big) \, 
}{
\vartheta_{11}(\tau, \, z + \nu\tau)}
\, = \, 
\frac{
q^{-\frac14(m+\frac12)(\frac{\nu}{m+\frac12})^2}
e^{-\pi i(m+\frac12) \frac{\nu}{m+\frac12} \cdot 2z} \, 
\theta_{\nu+\frac12, \, m+\frac12}^{(-)}(\tau, \, 2z)
}{
(-1)^{\nu} \, 
q^{-\frac12 \nu^2} \, 
e^{-2\pi i\nu z} \, \vartheta_{11}(\tau,z)}
\\[0mm]
&=&
- \, i \, (-1)^{\nu} \, 
q^{\frac{m}{2(m+\frac12)}\nu^2} \, 
\frac{\theta_{\nu+\frac12, m+\frac12}^{(-)}(\tau, \, 2z)}{
\theta_{\frac12, \frac12}^{(-)}(\tau, 2z)}
\end{eqnarray*}}
Substituting these, the equation \eqref{eindef3:qn:2023-1123b1}
becomes as follows:
{\allowdisplaybreaks
\begin{eqnarray*}
& & \hspace{4mm}
(-1)^{\nu} \, q^{-\frac{m^2\nu^2}{m+\frac12}} \, 
\theta^{(-)}_{\nu+\frac12, m+\frac12}(\tau,0) \, 
\Phi^{[m,0]}_1(\tau, \, z+\nu\tau, \, z-\nu\tau, \, 0)
\\[3mm]
& & 
+ \,\ 
q^{-\frac{m\nu}{2(m+\frac12)}} \, 
\sum_{\substack{r \, \in \, \zzz \\[1mm] 0 \,\leq \, r \, < \, 2m}} 
\Big[\sum_{\substack{j, \, p \, \in \, \zzz \\[1mm]
0 \, < \, p \, \leq \, j}} 
-
\sum_{\substack{j, \, p \, \in \, \zzz \\[1mm]
j \, < \, p \, \leq \, 0}} \Big] 
f_{-j, \, r-2pm} \, \theta_{r,m}(\tau, \, 2z)
\\[2mm]
&= &
- \,\ i \,\ \eta(\tau)^3 \, \times \, 
(- i) \, (-1)^{\nu} \, 
q^{\frac{m \nu^2}{2(m+\frac12)}} \, 
\frac{\theta_{\nu+\frac12, m+\frac12}^{(-)}(\tau, \, 2z)}{
\theta_{\frac12, \frac12}^{(-)}(\tau, 2z)}
\end{eqnarray*}}
namely
{\allowdisplaybreaks
\begin{eqnarray}
& & \hspace{-10mm}
(-1)^{\nu} \, q^{-m\nu^2} \, 
\theta^{(-)}_{\nu+\frac12, m+\frac12}(\tau,0)
\Phi^{[m,0]}_1(\tau, \, z+\nu\tau, \, z-\nu\tau, \, 0)
\nonumber
\\[3mm]
& & 
+ \,\ \underbrace{
q^{-\frac{m}{2(m+\frac12)}(\nu^2+\nu)} 
\sum_{\substack{r \, \in \, \zzz \\[1mm] 0 \,\leq \, r \, < \, 2m}} 
\Big[\sum_{\substack{j, \, p \, \in \, \zzz \\[1mm]
0 \, < \, p \, \leq \, j}} 
-
\sum_{\substack{j, \, p \, \in \, \zzz \\[1mm]
j \, < \, p \, \leq \, 0}} \Big] 
f_{-j, \, r-2pm} \, \theta_{r,m}(\tau, \, 2z)}_{\rm (I)}
\nonumber
\\[2mm]
&= &
- \,\ (-1)^{\nu} \, \eta(\tau)^3 \, 
\frac{\theta_{\nu+\frac12, m+\frac12}^{(-)}(\tau, \, 2z)}{
\theta_{0, \frac12}^{(-)}(\tau, 2z)}
\label{indef3:eqn:2023-1123b2}
\end{eqnarray}}
where $f_{-j, \, r-2pm}$ is computed as follows in the case 
$(a,b,c,d)=(\nu,-\nu,0,0)$: 
{\allowdisplaybreaks
\begin{eqnarray*}
& & \hspace{-10mm}
f_{-j, \, r-2pm}\big|_{(z_1,z_2)=(z+\nu\tau, z-\nu\tau)} 
\\[2mm]
&=&\Big[
(-1)^j \, 
q^{(m+\frac12)(-j+\frac{\frac12}{2(m+\frac12)})^2-\frac{1}{4m}(r-2pm)^2} \,
e^{[2\pi ijm+\pi i(r-2pm)](z_1-z_2)}
\Big]_{(z_1,z_2)=(z+\nu\tau, z-\nu\tau)} 
\\[2mm]
&=&
(-1)^j \, 
q^{(m+\frac12)(j-\frac{1}{4(m+\frac12)})^2-\frac{1}{4m}(r-2pm)^2} \,
q^{2jm\nu+(r-2pm)\nu} 
\\[2mm]
&=&
(-1)^j \,
q^{(m+\frac12)(j+\nu-\frac{\nu+\frac12}{2(m+\frac12)})^2
\, - \, m \, (p+\nu-\frac{r}{2m})^2
\, + \, \frac{m\nu(\nu+1)}{2(m+\frac12)}}
\end{eqnarray*}}
Then we have 
$$
{\rm (I)} =
\sum_{\substack{r \, \in \, \zzz \\[1mm] 0 \,\leq \, r \, < \, 2m}} 
\Big[\sum_{\substack{j, \, p \, \in \, \zzz \\[1mm]
0 \, < \, p \, \leq \, j}} 
-
\sum_{\substack{j, \, p \, \in \, \zzz \\[1mm]
j \, < \, p \, \leq \, 0}} \Big] 
(-1)^j \,
q^{(m+\frac12)(j+\nu-\frac{\nu+\frac12}{2(m+\frac12)})^2
\, - \, m \, (p+\nu-\frac{r}{2m})^2} 
\theta_{r,m}(\tau, \, 2z)
$$
so the formula \eqref{indef3:eqn:2023-1123b2} gives
{\allowdisplaybreaks
\begin{eqnarray}
& & \hspace{-7mm}
(-1)^{\nu} \, q^{-m\nu^2} \, 
\theta^{(-)}_{\nu+\frac12, m+\frac12}(\tau,0) \, 
\Phi^{[m,0]}_1(\tau, \, z+\nu\tau, \, z-\nu\tau, \, 0)
\nonumber
\\[3mm]
& & 
+ \,\ 
\sum_{\substack{r \, \in \, \zzz \\[1mm] 0 \,\leq \, r \, < \, 2m}} 
\Big[\sum_{\substack{j, \, p \, \in \, \zzz \\[1mm]
0 \, < \, p \, \leq \, j}} 
-
\sum_{\substack{j, \, p \, \in \, \zzz \\[1mm]
j \, < \, p \, \leq \, 0}} \Big] 
(-1)^j \,
q^{(m+\frac12)(j+\nu-\frac{\nu+\frac12}{2(m+\frac12)})^2
\, - \, m \, (p+\nu-\frac{r}{2m})^2} 
\theta_{r,m}(\tau, \, 2z)
\nonumber
\\[2mm]
&= &
- \,\ (-1)^{\nu} \, \eta(\tau)^3 \, 
\frac{\theta_{\nu+\frac12, m+\frac12}^{(-)}(\tau, \, 2z)}{
\theta_{\frac12, \frac12}^{(-)}(\tau, 2z)}
\label{indef3:eqn:2023-1123b3}
\end{eqnarray}}
\end{subequations}
Replacing $z$ with $\frac{z}{2}$ in \eqref{indef3:eqn:2023-1123b3},
we obtain the formula \eqref{indef3:eqn:2023-1123a},
proving Lemma \ref{indef3:lemma:2023-1123c}.
\end{proof}

\vspace{3mm}

For $m \in \frac12 \nnn$ and $n, \nu \in \zzz$ 
satisfying the conditions
\begin{equation}
0 \leq n <2m \hspace{10mm} \text{and} \hspace{10mm} 0 \leq \nu \leq 2m
\label{indef3:eqn:2023-1124a}
\end{equation} 
we define the functions 
$g^{[m]}_{n,\nu}(\tau)$ and $h^{[m]}_{n, \nu}(\tau,z)$ 
and $G^{\, [m]}_{n,\nu}(\tau,z)$ by the following formulas: 
%
\begin{subequations}
{\allowdisplaybreaks
\begin{eqnarray}
g^{[m]}_{n, \nu}(\tau) &:=& \hspace{-2mm}
\Big[\sum_{\substack{j, \, p \, \in \, \zzz \\[1mm]
0 \, < \, p \, \leq \, j}} 
-
\sum_{\substack{j, \, p \, \in \, \zzz \\[1mm]
j \, < \, p \, \leq \, 0}} \Big] 
(-1)^j \,
q^{(m+\frac12)(j+\nu-\frac{\nu+\frac12}{2(m+\frac12)})^2
\, - \, m \, (p+\nu-\frac{n}{2m})^2} 
\label{indef3:eqn:2023-1124b1}
\\[3mm]
h^{[m]}_{n, \nu}(\tau, z) &:=&
\sum_{j \in \zzz} \, 
\frac{e^{2\pi imjz+\pi inz} \, q^{mj^2+n(j+\nu)}
}{1-e^{2\pi imz} \, q^{2m(j+\nu)}}
\label{indef3:eqn:2023-1124b2}
\\[2mm]
G^{\, [m]}_{n, \nu}(\tau, z) &:=&
g^{[m]}_{n,\nu}(\tau) \, \theta_{n,m}(\tau,z)
\, + \, 
(-1)^{\nu} \, q^{-m\nu^2} \, 
\theta_{\nu+\frac12,m+\frac12}^{(-)}(\tau,0) \, 
h^{[m]}_{n, \nu}(\tau, z)
\label{indef3:eqn:2023-1124b3}
\end{eqnarray}}
\end{subequations}
We claim that the function $G^{\, [m]}_{n, \nu}(\tau, z)$ has the 
following expression:

\medskip

\begin{prop} 
\label{indef3:prop:2023-1124a}
For $m \in \frac12 \nnn$ and $n, \nu \in \zzz$ satisfying 
the condition \eqref{indef3:eqn:2023-1124a}, the following 
formulas hold:
\begin{enumerate}
\item[{\rm 1)}] \,\ ${\displaystyle 
G^{\, [m]}_{n, \nu}(\tau, z)
\,\ = \,\ 
- \, \frac{(-1)^{\nu}}{2m} \, \eta(\tau)^3 \, 
\sum_{\ell=0}^{2m-1} e^{\frac{\pi in\ell}{m}} \,\ 
\dfrac{\theta_{\nu+\frac12, m+\frac12}^{(-)}\Big(\tau, \, z-\dfrac{\ell}{m}\Big)
}{
\theta_{\frac12, \frac12}^{(-)}\Big(\tau, \, z-\dfrac{\ell}{m}\Big)}
}$
\begin{equation}
{ } \hspace{-15mm}
= \,\ 
- \, \frac{(-1)^{\nu}}{2m} \, \eta(\tau)^3 
\sum_{\ell \, \in \zzz/2m\zzz} 
e^{\frac{\pi in\ell}{m}} \,\ 
\dfrac{\theta_{\nu+\frac12, m+\frac12}^{(-)}\Big(\tau, \, z-\dfrac{\ell}{m}\Big)
}{
\theta_{\frac12, \frac12}^{(-)}\Big(\tau, \, z-\dfrac{\ell}{m}\Big)}
\label{indef3:eqn:2023-1124e1}
\end{equation}
\item[{\rm 2)}]  \,\ ${\displaystyle 
\sum_{n=0}^{2m-1}G^{\, [m]}_{n, \nu}(\tau, z)
\,\ = \,\ 
- \, (-1)^{\nu} \, \eta(\tau)^3 \, 
\dfrac{\theta_{\nu+\frac12, m+\frac12}^{(-)}(\tau, \, z)}{
\theta_{\frac12, \frac12}^{(-)}(\tau, \, z)}
}$
\end{enumerate}
\end{prop}

\begin{proof}
1) \,\ 
Letting $z \rightarrow z-\frac{\ell}{m}$ \,\ 
$(\ell \in \zzz \,\ {\rm s.t.} \,\ 0 \leq \ell < 2m)$ 
in \eqref{indef3:eqn:2023-1123a} and using Lemma 1.1 in \cite{W2022c},
we have 
{\allowdisplaybreaks
\begin{eqnarray}
& & \hspace{-7mm}
(-1)^{\nu} \, q^{-m\nu^2} \, 
\theta^{(-)}_{\nu+\frac12, m+\frac12}(\tau,0) \, 
\Phi^{[m,0]}_1\Big(\tau, \, 
\frac{z}{2}-\frac{\ell}{2m}+\nu\tau, \, 
\frac{z}{2}-\frac{\ell}{2m}-\nu\tau, \, 0\Big)
\nonumber
\\[3.5mm]
& & \hspace{-8mm}
+ \, e^{-\frac{\pi ir\ell}{m}}
\hspace{-3mm}
\sum_{\substack{r \, \in \, \zzz \\[1mm] 0 \,\leq \, r \, < \, 2m}} 
\Big[\sum_{\substack{j, \, p \, \in \, \zzz \\[1mm]
0 \, < \, p \, \leq \, j}} 
-
\sum_{\substack{j, \, p \, \in \, \zzz \\[1mm]
j \, < \, p \, \leq \, 0}} \Big] 
(-1)^j \,
q^{(m+\frac12)(j+\nu-\frac{\nu+\frac12}{2(m+\frac12)})^2
- \, m \, (p+\nu-\frac{r}{2m})^2} 
\theta_{r,m}(\tau, z)
\nonumber
\\[2mm]
& &\hspace{-10mm}
= \,\ - \,\ (-1)^{\nu} \, \eta(\tau)^3 \, \frac{
\theta_{\nu+\frac12, m+\frac12}^{(-)}\Big(\tau, \, z-\dfrac{\ell}{m}\Big)}{
\theta_{\frac12, \frac12}^{(-)}\Big(\tau, \, z-\dfrac{\ell}{m}\Big)}
\label{indef3:eqn:2023-1124c}
\end{eqnarray}}
Making the sum $\sum_{\ell=0}^{2m-1}e^{\frac{\pi in\ell}{m}} \times 
\eqref{indef3:eqn:2023-1124c}$, for each $n \in \zzz$ such that 
$0 \leq n <2m$, we have
\begin{subequations}
{\allowdisplaybreaks
\begin{eqnarray}
& & \hspace{-7mm}
(-1)^{\nu} \, q^{-m\nu^2} \, 
\theta^{(-)}_{\nu+\frac12, m+\frac12}(\tau,0) \, 
\underbrace{\sum_{\ell=0}^{2m-1}e^{\frac{\pi in\ell}{m}}
\Phi^{[m,0]}_1\Big(\tau, \,\ 
\frac{z}{2}-\frac{\ell}{2m}+\nu\tau, \,\ 
\frac{z}{2}-\frac{\ell}{2m}-\nu\tau, \,\ 0\Big)}_{(A)}
\nonumber
\\[1mm]
& & \hspace{-8mm}
+ \, 2m \, 
\underbrace{\Big[\sum_{\substack{j, \, p \, \in \, \zzz \\[1mm]
0 \, < \, p \, \leq \, j}} 
-
\sum_{\substack{j, \, p \, \in \, \zzz \\[1mm]
j \, < \, p \, \leq \, 0}} \Big] 
(-1)^j \,
q^{(m+\frac12)(j+\nu-\frac{\nu+\frac12}{2(m+\frac12)})^2
\, - \, m \, (p+\nu-\frac{n}{2m})^2} }_{\substack{|| \\[0mm] 
{\displaystyle \hspace{-4mm}
g^{[m]}_{n, \nu}(\tau)
}}}
\theta_{n,m}(\tau, z)
\nonumber
\\[2mm]
& &\hspace{-10mm}
= \,\ - \,\ (-1)^{\nu} \, \eta(\tau)^3 \, 
\sum_{\ell=0}^{2m-1} e^{\frac{\pi in\ell}{m}} \,\ \frac{
\theta_{\nu+\frac12, m+\frac12}^{(-)}\Big(\tau, \, z-\dfrac{\ell}{m}\Big)}{
\theta_{\frac12, \frac12}^{(-)}\Big(\tau, \, z-\dfrac{\ell}{m}\Big)}
\label{indef3:eqn:2023-1124d1}
\end{eqnarray}}
since $\sum_{\ell=0}^{2m-1}e^{\frac{\pi i(n-r)\ell}{m}}
\,\ = \, \left\{
\begin{array}{ccl}
0 & {\rm if} & n-r \ne 0
\\[1mm]
2m & {\rm if} & n-r = 0
\end{array}\right. . $ 
Here $(A)$ is computed by letting $
(z_1, z_2)=(\frac{z}{2}+\nu\tau, \frac{z}{2}-\nu\tau) $ and 
$p \rightarrow n$ in the formula 1) of Lemma 
\ref{indef3:lemma:2023-1122b} as follows:
$$
(A) \, = \, 
2m \sum_{j \in \zzz} 
\frac{e^{2\pi imjz+2\pi in(\frac{z}{2}+\nu\tau)}q^{mj^2+nj}
}{1-e^{4\pi im(\frac{z}{2}+\nu\tau)} \, q^{2mj}}
\,\ = \,\ 
2m \sum_{j \in \zzz} 
\frac{e^{2\pi imjz+\pi inz} 
q^{mj^2+n(j+\nu)}
}{1-e^{2\pi imz} \, q^{2m(j+\nu)}}
$$
Substituting this into \eqref{indef3:eqn:2023-1124d1}, we have 
%
{\allowdisplaybreaks
\begin{eqnarray}
& & \hspace{-7mm}
2m \, 
(-1)^{\nu} \, q^{-m\nu^2} \, 
\theta^{(-)}_{\nu+\frac12, m+\frac12}(\tau,0) \, 
\sum_{j \in \zzz} \, 
\frac{e^{2\pi imjz+\pi inz} \, q^{mj^2+n(j+\nu)}
}{1-e^{2\pi imz} \, q^{2m(j+\nu)}}
\nonumber
\\[3.5mm]
& & 
+ \,\ 2m \, g^{[m]}(\tau) \, \theta_{n,m}(\tau, z)
\nonumber
\\[2mm]
& &\hspace{-10mm}
= \,\ - \,\ (-1)^{\nu} \, \eta(\tau)^3 \, 
\sum_{\ell=0}^{2m-1} e^{\frac{\pi in\ell}{m}} \,\ 
\frac{\theta_{\nu+\frac12, m+\frac12}^{(-)}\Big(\tau, \, z-\dfrac{\ell}{m}\Big)
}{
\theta_{\frac12, \frac12}^{(-)}\Big(\tau, \, z-\dfrac{\ell}{m}\Big)}
\label{indef3:eqn:2023-1124d2}
\end{eqnarray}}
\end{subequations}
proving 1). \,\ 
2) is obtained immediately from 1) since 
{\allowdisplaybreaks
\begin{eqnarray*}
& & \hspace{-10mm}
\sum_{n=0}^{2m-1}\sum_{\ell=0}^{2m-1} 
e^{\frac{\pi in\ell}{m}} \,\ 
\frac{\theta_{\nu+\frac12, m+\frac12}^{(-)}\Big(\tau, z-\dfrac{\ell}{m}\Big)
}{
\theta_{\ast, \frac12}^{(-)}\Big(\tau, z-\dfrac{\ell}{m}\Big)}
\, = 
\sum_{\ell=0}^{2m-1} 
\Bigg(
\underbrace{\sum_{n=0}^{2m-1}e^{\frac{\pi in\ell}{m}}
}_{2m \, \delta_{\ell,0}}\Bigg) 
\frac{\theta_{\nu+\frac12, m+\frac12}^{(-)}\Big(\tau, z-\dfrac{\ell}{m}\Big)}{
\theta_{\ast, \frac12}^{(-)}\Big(\tau, z-\dfrac{\ell}{m}\Big)}
\\[2mm]
& &
= \,\ 
2m \frac{\theta_{\nu+\frac12, m+\frac12}^{(-)}(\tau, z)}{
\theta_{\ast, \frac12}^{(-)}(\tau, z)} \hspace{10mm}
{\rm for} \quad \ast \, = \, 0, \, \tfrac12
\end{eqnarray*}}

\vspace{-10mm}
\end{proof}

\medskip

\begin{note} 
\label{indef3:note:2023-1124a}
For $m \in \frac12 \nnn$ and $n, \nu \in \zzz$ satisfying 
the condition \eqref{indef3:eqn:2023-1124a}, the following 
formula holds:
\begin{equation} 
G^{\, [m]}_{n, \nu}\Big(\tau, z+\frac{c}{m}\Big)
\,\ = \,\ 
e^{\frac{\pi inc}{m}} \, G^{\, [m]}_{n, \nu}(\tau, z) 
\hspace{10mm} ({}^{\forall} c \, \in \, \zzz)
\label{indef3:eqn:2023-1124f1}
\end{equation}
\end{note}

\begin{proof} \,\ 
Letting \, $z \rightarrow z+\frac{c}{m}$ \, 
in \eqref{indef3:eqn:2023-1124e1}, we have 
$$
G^{\, [m]}_{n, \nu}\Big(\tau,z+\frac{c}{m}\Big) 
\,\ = \,\ 
- \, \frac{(-1)^{\nu}}{2m} \, \eta(\tau)^3 \, 
\sum_{\ell \, \in \, \zzz/2m\zzz}
e^{\frac{\pi in\ell}{m}} \,\ \frac{
\theta_{\nu+\frac12, m+\frac12}^{(-)}
\Big(\tau, \, z-\dfrac{\ell-c}{m}\Big)
}{
\theta_{\frac12, \frac12}^{(-)}\Big(\tau, \, z-\dfrac{\ell-c}{m}\Big)}
$$
Putting $\ell-c=:\ell'$, this equation becomes
$$
= \, 
- \, \frac{(-1)^{\nu}}{2m} \, \eta(\tau)^3 \, 
\sum_{\ell' \, \in \, \zzz/2m\zzz}
e^{\frac{\pi in(\ell'+c)}{m}} \,\ \frac{
\theta_{\nu+\frac12, m+\frac12}^{(-)}\Big(\tau, \, z-\dfrac{\ell'}{m}\Big)
}{
\theta_{\frac12, \frac12}^{(-)}\Big(\tau, \, z-\dfrac{\ell'}{m}\Big)}
\,\ = \,\ 
e^{\frac{\pi inc}{m}} \, G^{\, [m]}_{n, \nu}(\tau,z)
$$
proving the formula \eqref{indef3:eqn:2023-1124f1}.
\end{proof}

\medskip



\begin{lemma} 
\label{indef3:lemma:2023-1125a}
Let $m \in \frac12 \nnn$ and $n$ and $\nu$ be no-negative integers satisfying 
the condition \eqref{indef3:eqn:2023-1124a}. Then
\begin{enumerate}
\item[{\rm 1)}] the $S$-transformation of the function 
$G^{\, [m]}_{n, \nu}(\tau,z)$ is given by the folllowing formula:
\begin{subequations}
{\allowdisplaybreaks
\begin{eqnarray}
& & \hspace{-15mm}
G^{\, [m]}_{n,\nu}\Big(-\frac{1}{\tau}, \frac{z}{\tau}\Big)
\,\ = \,\
i \,\ \frac{(-i\tau)^{\frac32}}{2m\sqrt{2m+1}}
\sum_{n'=0}^{2m-1}
\sum_{\nu'=0}^{2m}
\sum_{\ell=0}^{2m-1} 
\nonumber
\\[2mm]
& & \hspace{-5mm}
\times \,\ 
(-1)^{\nu+k} \, 
e^{\frac{\pi in\ell}{m}} \, 
e^{-\frac{\pi i}{m+\frac12}(\nu+\frac12)(\nu'+\frac12)} \, 
e^{\frac{\pi im}{2\tau}z^2} \, e^{-\pi i\ell z} \, 
q^{\frac{\ell^2}{4m}} \, 
G^{\, [m]}_{n',\nu'}\Big(\tau,z-\frac{\ell \tau}{m}\Big)
\label{indef3:eqn:2023-1125a1}
\\[2mm]
& & \hspace{-10mm}
= \,\ 
i \,\ \frac{(-i\tau)^{\frac32}}{2m\sqrt{2m+1}}
\sum_{n'=0}^{2m-1}
\sum_{\nu'=0}^{2m} \, 
\sum_{\ell \in \zzz/2m\zzz} 
\nonumber
\\[2mm]
& & \hspace{-5mm}
\times \,\ 
(-1)^{\nu+\nu'} \, 
e^{\frac{\pi in\ell}{m}} \, 
e^{-\frac{\pi i}{m+\frac12}(\nu+\frac12)(\nu'+\frac12)} \, 
e^{\frac{\pi im}{2\tau}z^2} \, e^{-\pi i\ell z} \, 
q^{\frac{\ell^2}{4m}} \, 
G^{\, [m]}_{n',\nu'}\Big(\tau,z-\frac{\ell \tau}{m}\Big)
\label{indef3:eqn:2023-1125a2}
\end{eqnarray}}
\end{subequations}
\item[{\rm 2)}] In the case $m \in \nnn$, the $T$-transformation 
of the function $G^{\, [m]}_{n, \nu}(\tau,z)$ is given by 
the folllowing formula:
\begin{equation}
G^{\, [m]}_{n, \nu}(\tau+1, z)
\, = \,\ 
e^{\frac{\pi i}{2(m+\frac12)}(\nu+\frac12)^2} \, 
G^{\, [m]}_{n, \nu}(\tau, z)
\label{indef3:eqn:2023-1125a3}
\end{equation}
\end{enumerate}
\end{lemma}

\begin{proof} 1) Letting 
$(\tau,z) \rightarrow (-\frac{1}{\tau}, \frac{z}{\tau})$ 
in \eqref{indef3:eqn:2023-1124e1}, one has
\begin{equation}
G^{\, [m]}_{n, \nu}\Big(-\frac{1}{\tau}, \frac{z}{\tau}\Big) 
\,\ = \,\ 
- \, \frac{(-1)^{\nu}}{2m} \, 
\eta\Big(-\dfrac{1}{\tau}\Big)^3 \, 
\sum\limits_{\ell=0}^{2m-1} e^{\frac{\pi in\ell}{m}} \,\ 
\dfrac{\theta_{\nu+\frac12, m+\frac12}^{(-)}
\Big(-\dfrac{1}{\tau}, \, \dfrac{z}{\tau}-\dfrac{\ell}{m}\Big)}{
\theta_{\frac12, \frac12}^{(-)}
\Big(-\dfrac{1}{\tau}, \, \dfrac{z}{\tau}-\dfrac{\ell}{m}\Big)}
\label{indef3:eqn:2023-1125b}
\end{equation}
Noticing, by using Lemma 1.3 in \cite{W2022c}, that
{\allowdisplaybreaks
\begin{eqnarray*}
& & \hspace{-10mm}
\eta(-\tfrac{1}{\tau}) \,\ = \,\ (-i\tau)^{\frac12} \, \eta(\tau) 
\\[2mm]
& & \hspace{-10mm}
\theta_{\nu+\frac12, m+\frac12}^{(-)}
\Big(-\frac{1}{\tau}, \, \frac{z}{\tau}-\frac{\ell}{m}\Big)
\,\ = \,\ 
\frac{(-i\tau)^{\frac12}}{\sqrt{2(m+\frac12)}} \, 
e^{\frac{\pi i}{2\tau}(m+\frac12)z^2} \, 
e^{-\frac{\pi i(m+\frac12)\ell}{m}z} \, 
q^{\frac{\ell^2}{4m^2}(m+\frac12)}
\\[2mm]
& & \hspace{10mm}
\times \,\ 
\sum_{k=0}^{2m} 
e^{-\frac{\pi i(\nu+\frac12) (k+\frac12)}{m+\frac12}} \, 
\theta_{k+\frac12,m+\frac12}^{(-)}\Big(\tau,z-\dfrac{\ell \, \tau}{m}\Big)
\\[2mm]
& & \hspace{-10mm}
\theta_{\frac12, \frac12}^{(-)}
\Big(-\frac{1}{\tau}, \, \frac{z}{\tau}-\frac{\ell}{m}\Big)
= 
- \, i \, (-i\tau)^{\frac12} \, 
e^{\frac{\pi i}{4\tau}z^2} \, 
e^{-\frac{\pi i\ell}{2m}z} 
q^{\frac{\ell^2}{8m^2}} \, 
\theta_{\frac12, \frac12}^{(-)}\Big(\tau,z-\frac{\ell \, \tau}{m}\Big) \, ,
\end{eqnarray*}}
the RHS of the above equation \eqref{indef3:eqn:2023-1125b}
is rewritten as follows:
{\allowdisplaybreaks
\begin{eqnarray*}
& & \hspace{-15mm}
\text{RHS of \eqref{indef3:eqn:2023-1125b}} = \,\ 
- \, \frac{(-1)^{\nu}}{2m} \, 
\eta\Big(-\frac{1}{\tau}\Big)^3 \, 
\sum_{\ell=0}^{2m-1} e^{\frac{\pi in\ell}{m}} \,\ 
\frac{\theta_{\nu+\frac12, m+\frac12}^{(-)}
\Big(-\dfrac{1}{\tau}, \, \dfrac{z}{\tau}-\dfrac{\ell}{m}\Big)}{
\theta_{\frac12, \frac12}^{(-)}
\Big(-\dfrac{1}{\tau}, \, \dfrac{z}{\tau}-\dfrac{\ell}{m}\Big)}
\\[2.5mm]
&=&
- \, i \, (-1)^{\nu} \, 
\frac{(-i\tau)^{\frac32} }{2m\sqrt{2m+1}} \, 
\sum_{\ell=0}^{2m-1} e^{\frac{\pi in\ell}{m}} \, 
e^{\frac{\pi im}{2\tau}z^2} \, e^{-\pi i\ell z} \, 
q^{\frac{\ell^2}{4m}}
\\[2mm]
& &
\times \, \sum_{k=0}^{2m}
e^{-\frac{\pi i}{m+\frac12}(\nu+\frac12)(k+\frac12)}
\eta(\tau)^3 \,\ 
\frac{\theta_{k+\frac12,m+\frac12}^{(-)}(\tau,z-\frac{\ell \tau}{m})
}{
\theta_{\frac12,\frac12}^{(-)}(\tau,z-\frac{\ell \tau}{m})}
\end{eqnarray*}}
which is equal to 
{\allowdisplaybreaks
\begin{eqnarray*}
&=&
i \,\ \frac{(-i\tau)^{\frac32}}{2m\sqrt{2m+1}}
\sum_{n'=0}^{2m-1}
\sum_{k=0}^{2m}
\sum_{\ell=0}^{2m-1} 
\\[2mm]
& & 
\times \,\ 
(-1)^{\nu+k} \, 
e^{\frac{\pi in\ell}{m}} \, 
e^{-\frac{\pi i}{m+\frac12}(\nu+\frac12)(k+\frac12)} \, 
e^{\frac{\pi im}{2\tau}z^2} \, e^{-\pi i\ell z} \, 
q^{\frac{\ell^2}{4m}} \, 
G^{\, [m]}_{n',k}\Big(\tau,z-\frac{\ell \tau}{m}\Big)
\end{eqnarray*}}
by Proposition \ref{indef3:prop:2023-1124a}, proving the formula 
\eqref{indef3:eqn:2023-1125a1}. \,\ 
2) is obtained easily from \eqref{indef3:eqn:2023-1124e1} and 
Lemma 1.4 in \cite{W2022c}.
\end{proof}

\medskip

We note that functions $h^{[m]}_{n,\nu}(\tau,z)$ and 
$G^{\, [m]}_{n,\nu}(\tau,z)$ have the following properties.

\medskip

\begin{note} 
\label{indef3:note:2023-1127a}
For $m \in \frac12 \nnn$ and $n, \nu \in \zzz$ satisfying 
the condition \eqref{indef3:eqn:2023-1124a} and $a,b,a',b' \in \zzz$, 
the following formulas hold:
\begin{enumerate}
\item[{\rm 1)}] \,\ $G^{\, [m]}_{n, \nu}(\tau, \, z+2a\tau+2b)
\,\ = \,\ 
e^{2\pi inb} \, e^{-2\pi imaz} \, q^{-ma^2} \, G^{\, [m]}_{n, \nu}(\tau, z)$
\item[{\rm 2)}] \,\ If \, $a-a' \in 2m\zzz$ \, and \, $b-b' \in 2m \zzz$, then 
\begin{equation}
e^{\pi iaz} \, q^{\frac{a^2}{4m}} \, 
G^{\, [m]}_{n, \nu}\Big(\tau, z+\frac{a\tau}{m}+\frac{b}{m}\Big)
\, = \, 
e^{\pi ia'z} \, q^{\frac{a'{}^2}{4m}} \,
G^{\, [m]}_{n, \nu}\Big(\tau, z+\frac{a'\tau}{m}+\frac{b'}{m}\Big) 
\label{indef3:eqn:2023-1127a1}
\end{equation}
\end{enumerate}
\end{note}

\begin{proof} 1) is obtained easily from \eqref{indef3:eqn:2023-1124e1} 
by calculation using Lemma 1.1 in \cite{W2022c}.
2) follows immediately from 1).
\end{proof}

\medskip


\begin{note} 
\label{indef3:note:2023-1127b}
Let $m \in \frac12 \nnn$ and $n,n',\nu,\nu'$ be integers 
satisfying the conditions
\begin{equation}
0 \leq n, n' <2m, \quad 0 \leq \nu, \nu' \leq 2m \quad 
\text{and} \quad 
n+n'\in 2m \zzz \quad \nu+\nu'=2m
\label{indef3:eqn:2023-1127b1}
\end{equation}
Then
\begin{equation} \hspace{-30mm}
G^{\, [m]}_{n', \nu'}(\tau, -z)
\,\ = \,\ 
(-1)^{2m} \, G^{\, [m]}_{n, \nu}(\tau, z)
\label{indef3:eqn:2023-1127b2}
\end{equation}
\end{note}

\begin{proof} First we consider the case $n+n'=2m$. 
Since $\nu'+\frac12 \, = \, 2m+1-(\nu+\frac12)$, the formula 
\eqref{indef3:eqn:2023-1124e1} gives the following:
{\allowdisplaybreaks
\begin{eqnarray*}
& & \hspace{-10mm}
G^{\, [m]}_{n', \nu'}(\tau, -z)
\,\ = \,\ 
- \, (-1)^{\nu'} \, \eta(\tau)^3 \, 
\sum_{\ell=0}^{2m-1} e^{\frac{\pi in'\ell}{m}} \,\ 
\dfrac{\theta_{\nu'+\frac12, m+\frac12}^{(-)}\Big(\tau, \, -z-\dfrac{\ell}{m}\Big)
}{
\theta_{\frac12, \frac12}^{(-)}\Big(\tau, \, -z-\dfrac{\ell}{m}\Big)}
\\[2mm]
&=&
- \, (-1)^{2m-\nu} \, \eta(\tau)^3 \, 
\sum_{\ell=0}^{2m-1} 
e^{\frac{\pi i(2m-n)\ell}{m}}
\frac{\theta_{2(m+\frac12)-(\nu+\frac12), m+\frac12}^{(-)}
\Big(\tau, \, -z-\dfrac{\ell}{m}\Big)
}{\theta_{\frac12, \frac12}^{(-)}\Big(\tau, \, -z-\dfrac{\ell}{m}\Big)}
\end{eqnarray*}}
which is, by using Lemma 1.2 in \cite{W2022c}, rewritten as
$$
= \,\ 
- \, (-1)^{2m+\nu} \, \eta(\tau)^3 \, 
\sum_{\ell=0}^{2m-1} \, 
e^{-\frac{\pi in\ell}{m}} \,\ 
\frac{\theta_{\nu+\frac12, m+\frac12}^{(-)}
\Big(\tau, \, z+\dfrac{\ell}{m}\Big)
}{\theta_{\frac12, \frac12}^{(-)}\Big(\tau, \, z+\dfrac{\ell}{m}\Big)}
\,\ = \,\ 
(-1)^{2m} \, G^{\, [m]}_{n, \nu}(\tau, z)
$$
proving the formula \eqref{indef3:eqn:2023-1127b2} in the case 
$n+n'=2m$. 
The proof in the case $n=n'=0$ is quite similar. 
\end{proof}

\section{Functions $F^{\, [m](a,b)}_{n, \nu}(\tau,z)$ }


For $m \in \frac12 \nnn$ and $n, \nu \in \zzz$ satisfying 
the condition \eqref{indef3:eqn:2023-1124a} and $a, b \in \zzz$, 
we define the function $F^{\, [m](a,b)}_{n, \nu}(\tau,z)$ by
\begin{equation}
F^{\, [m](a,b)}_{n, \nu}(\tau,z)
\,\ := \,\ 
e^{\pi iaz} \, q^{\frac{1}{4m}a^2} \, 
G^{\, [m]}_{n, \nu}
\Big(\tau, \, z+\frac{a\tau}{m}+\frac{b}{m}\Big)
\label{indef3:eqn:2023-1125c}
\end{equation}

\medskip

\begin{note} 
\label{indef3:note:2023-1128a}
Let $m \in \frac12 \nnn$ and $n$ and $\nu$ be integers satisfying 
the condition \eqref{indef3:eqn:2023-1124a}. Then 
\begin{enumerate}
\item[{\rm 1)}] \, the following formulas hold for $a, b, c \in \zzz$.
\begin{enumerate}
\item[{\rm (i)}] \,\ $F^{\, [m](a,b+c)}_{n,\nu}(\tau,z)
\,\ = \,\ 
e^{\frac{\pi inc}{m}} \, F^{\, [m](a,b)}_{n,\nu}(\tau,z)$
\item[{\rm (ii)}] \,\ $F^{\, [m](a,b)}_{n,\nu}(\tau,z)
\,\ = \,\ 
e^{\frac{\pi inb}{m}} \, F^{\, [m](a,0)}_{n,\nu}(\tau,z)$
\end{enumerate}
\item[{\rm 2)}] \,\ If integers $a, b, a', b'$ satisfy 
$a-a', \, b-b' \in 2m\zzz$, then 
$$
F^{\, [m](a,b)}_{n,\nu}(\tau,z)
\,\ = \,\ F^{\, [m](a',b')}_{n,\nu}(\tau,z)
$$
Namely \, $F^{\, [m](a,b)}_{n,\nu}(\tau,z)$ \, is defined for
$(a, b) \in (\zzz/2m\zzz)^2$.
\end{enumerate}
\end{note}

\begin{proof} 1) (resp. 2)) follows immediately 
from Note \ref{indef3:note:2023-1124a} 
(resp. Note \ref{indef3:note:2023-1127a}).
\end{proof}

\medskip



\begin{prop}  
\label{indef3:prop:2023-1128a}
Let $m \in \frac12 \nnn$ and $n$ and $\nu$ be integers satisfying 
the condition \eqref{indef3:eqn:2023-1124a} and 
$(a, b) \in (\zzz/2m\zzz)^2$. Then
\begin{enumerate}
\item[{\rm 1)}]  \,\ the $S$-transformation of 
$F^{\, [m](a,b)}_{n, \nu}(\tau,z)$ is given by the 
following formula:
\begin{subequations}
\begin{enumerate}
\item[{\rm (i)}] \,\ $F^{\, [m](a,b)}_{n,\nu}
\Big(-\dfrac{1}{\tau}, \dfrac{z}{\tau}\Big)
\, = \,\ 
\dfrac{i \, (-i\tau)^{\frac32}}{2m\sqrt{2m+1}} \, 
e^{\frac{\pi iaz}{\tau}} \, 
e^{-\frac{\pi ia^2}{2m\tau}} \, 
e^{\frac{\pi im}{2\tau}(z-\frac{a}{m})^2} 
\sum\limits_{n'=0}^{2m-1}
\sum\limits_{\nu'=0}^{2m} \, 
\sum\limits_{\ell \, \in \, \zzz/2m\zzz} $
\begin{equation} 
{} \hspace{-10mm}
\times \,\ 
(-1)^{\nu+\nu'} \, 
e^{-\frac{2\pi i}{2m+1}(\nu+\frac12)(\nu'+\frac12)} \, 
e^{\frac{\pi in(b-\ell)}{m}} \, 
e^{-\frac{\pi i\ell a}{m}} \, 
F^{\, [m](\ell, -a)}_{n',\nu'}(\tau,z)
\label{indef3:eqn:2023-1128a1}
\end{equation}
\item[{\rm (ii)}] \,\ $F^{\, [m](a,0)}_{n,\nu}
\Big(-\dfrac{1}{\tau}, \dfrac{z}{\tau}\Big)
\, = \,\ 
\dfrac{i \, (-i\tau)^{\frac32}}{2m\sqrt{2m+1}} \, 
e^{\frac{\pi iaz}{\tau}} \, 
e^{-\frac{\pi ia^2}{2m\tau}} \, 
e^{\frac{\pi im}{2\tau}(z-\frac{a}{m})^2} 
\sum\limits_{n'=0}^{2m-1}
\sum\limits_{\nu'=0}^{2m} \, 
\sum\limits_{\ell \, \in \, \zzz/2m\zzz} $
\begin{equation} 
{} \hspace{-10mm}
\times \,\ 
(-1)^{\nu+\nu'} \, 
e^{-\frac{2\pi i}{2m+1}(\nu+\frac12)(\nu'+\frac12)} \, 
e^{-\frac{\pi in'a}{m}} \, 
e^{-\frac{\pi i(n+a)\ell}{m}} \, 
F^{\, [m](\ell, 0)}_{n',\nu'}(\tau,z)
\label{indef3:eqn:2023-1128a2}
\end{equation}
\end{enumerate}
\end{subequations}
\item[{\rm 2)}]  \,\ In particular, the $S$-transformation of 
$F^{\, [m](a,0)}_{n, \nu}(\tau,0)$ is given as follows:
{\allowdisplaybreaks
\begin{eqnarray}
& & \hspace{-20mm}
F^{\, [m](a,0)}_{n,\nu}
\Big(-\dfrac{1}{\tau}, 0\Big)
\, = \,\ 
\dfrac{i \, (-i\tau)^{\frac32}}{2m\sqrt{2m+1}} \, 
\sum\limits_{n'=0}^{2m-1}
\sum\limits_{\nu'=0}^{2m} \, 
\sum\limits_{\ell \, \in \, \zzz/2m\zzz} 
\nonumber
\\[2mm]
& &
{} \hspace{-10mm}
\times \,\ 
(-1)^{\nu+\nu'} \, 
e^{-\frac{2\pi i}{2m+1}(\nu+\frac12)(\nu'+\frac12)} \, 
e^{-\frac{\pi in'a}{m}} \, 
e^{-\frac{\pi i(n+a)\ell}{m}} \, 
F^{\, [m](\ell, 0)}_{n',\nu'}(\tau,0)
\label{indef3:eqn:2023-1128a3}
\end{eqnarray}}
\item[{\rm 3)}]  \,\ In the case $m \in \nnn$, the 
$T$-transformation of $F^{\, [m](a,b)}_{n, \nu}(\tau,z)$ is given 
by the following formula:
\begin{equation}
F^{\, [m](a,b)}_{n,\nu}(\tau+1,z)
\,\ = \,\ 
e^{\frac{\pi i}{2m}[(n+a)^2-n^2] \, + \, 
\frac{\pi i}{2(m+\frac12)}(\nu+\frac12)^2} \, 
F^{\, [m](a,b)}_{n,\nu}(\tau,z)
\label{indef3:eqn:2023-1128a4}
\end{equation}
\end{enumerate}
\end{prop}

\begin{proof} 1) \,\ By the $S$-transformation formula 
\eqref{indef3:eqn:2023-1125a2} of $G^{\, [m]}_{n,\nu}$
in Lemma \ref{indef3:lemma:2023-1125a}, one has
{\allowdisplaybreaks
\begin{eqnarray*}
& & \hspace{-10mm}
F^{\, [m](a,b)}_{n,\nu}
\Big(-\frac{1}{\tau}, \frac{z}{\tau}\Big)
\,\ = \,\ 
e^{\frac{\pi iaz}{\tau}} \, 
e^{-\frac{2\pi i}{\tau} \cdot \frac{a^2}{4m}} \, 
G^{\, [m]}_{n,\nu}\Big(-\frac{1}{\tau}, \, 
\frac{z}{\tau}-\frac{a}{m\tau}+\frac{b}{m}
\Big)
\nonumber
\\[2mm]
& & \hspace{-10mm}
= \,\ 
e^{\frac{\pi iaz}{\tau}} \, 
e^{-\frac{\pi ia^2}{2m\tau}} \, 
\times \, 
\frac{i \, (-i\tau)^{\frac32}}{2m\sqrt{2m+1}}
\sum_{n'=0}^{2m-1}
\sum_{\nu'=0}^{2m} \, 
\sum_{\ell \in \zzz/2m\zzz} 
(-1)^{\nu+\nu'} \, 
e^{\frac{\pi in\ell}{m}} \, 
e^{-\frac{\pi i}{m+\frac12}(\nu+\frac12)(\nu'+\frac12)} \, 
\nonumber
\\[2mm]
& & 
\times \,\ 
e^{\frac{\pi im}{2\tau}(z-\frac{a}{m}+\frac{b\tau}{m})^2}
 \, 
e^{-\pi i\ell (z-\frac{a}{m}+\frac{b\tau}{m})} \, 
q^{\frac{\ell^2}{4m}} \, 
G^{\, [m]}_{n',\nu'}\Big(\tau,
z-\frac{a}{m}+\frac{b\tau}{m}-\frac{\ell \tau}{m}\Big)
\\[2mm]
& & \hspace{-10mm}
= \, 
\frac{i \, (-i\tau)^{\frac32}}{2m\sqrt{2m+1}} \, 
e^{\frac{\pi iaz}{\tau}} \, 
e^{-\frac{\pi ia^2}{2m\tau}} 
\sum_{n'=0}^{2m-1}
\sum_{\nu'=0}^{2m} \, 
\sum_{\ell \in \zzz/2m\zzz} 
(-1)^{\nu+\nu'} \, 
e^{\frac{\pi in\ell}{m}} \, 
e^{-\frac{\pi i}{m+\frac12}(\nu+\frac12)(\nu'+\frac12)}
\nonumber
\\[2mm]
& & 
\times \,\ 
e^{\frac{\pi i(\ell-b) a}{m}} \, 
e^{\frac{\pi im}{2\tau}(z-\frac{a}{m})^2} \, 
\underbrace{e^{\pi i(b-\ell)z} \, 
q^{\frac{1}{4m}(b-\ell)^2} \, 
G^{\, [m]}_{n',\nu'}\Big(\tau,
z+\frac{(b-\ell)\tau}{m}-\frac{a}{m}\Big)}_{\substack{|| \\[0mm] 
{\displaystyle F^{\, [m](b-\ell, -a)}_{n',\nu'}(\tau,z)
}}}
\end{eqnarray*}}
Putting $b-\ell=:\ell'$, this becomes
{\allowdisplaybreaks
\begin{eqnarray*}
& & \hspace{-10mm}
= \,\ 
\frac{i \, (-i\tau)^{\frac32}}{2m\sqrt{2m+1}} \,\ 
e^{\frac{\pi iaz}{\tau}} \, 
e^{-\frac{\pi ia^2}{2m\tau}} 
\sum_{n'=0}^{2m-1}
\sum_{\nu'=0}^{2m}
\sum_{\ell' \in \zzz/2m\zzz} 
(-1)^{\nu+\nu'} \, 
\nonumber
\\[2mm]
& & 
\times \,\ 
e^{\frac{\pi in(b-\ell')}{m}} \, 
e^{-\frac{\pi i\ell' a}{m}} \, 
e^{-\frac{2\pi i}{2m+1}(\nu+\frac12)(\nu'+\frac12)} \, 
e^{\frac{\pi im}{2\tau}(z-\frac{a}{m})^2} \, 
F^{\, [m](\ell', -a)}_{n',\nu'}(\tau,z)
\end{eqnarray*}}
proving \eqref{indef3:eqn:2023-1128a1}. 
The formula \eqref{indef3:eqn:2023-1128a2} is obtained from 
\eqref{indef3:eqn:2023-1128a1} and Note 
\ref{indef3:note:2023-1128a}. \, 2) follows immediately from 
\eqref{indef3:eqn:2023-1128a2}.

\medskip

\noindent 
3) \,\ By the $T$-transformation formula 
\eqref{indef3:eqn:2023-1125a3} of $G^{\, [m]}_{n,\nu}$
in Lemma \ref{indef3:lemma:2023-1125a}, one has
{\allowdisplaybreaks
\begin{eqnarray*}
& & \hspace{-15mm}
F^{\, [m](a,b)}_{n,\nu}(\tau+1,z) 
\, = \, 
e^{\pi iaz} \, 
e^{2\pi i (\tau+1)\frac{a^2}{4m}} \, 
G^{\, [m]}_{n, \nu}
\Big(\tau+1, \, z+\dfrac{a}{m}(\tau+1)+\dfrac{b}{m}\Big)
\\[2mm]
&=&
e^{\frac{\pi ia^2}{2m}} \, 
e^{\frac{\pi i}{2(m+\frac12)}(\nu+\frac12)^2} \hspace{-13mm} %
\underbrace{e^{\pi iaz} \, q^{\frac{a^2}{4m}} \, 
G^{\, [m]}_{n, \nu}
\Big(\tau, \, z+\dfrac{a\tau}{m}+\dfrac{b+a}{m}\Big)}_{\substack{|| 
\\[-1mm] {\displaystyle \hspace{20mm} %
F^{\, [m](a,b+a)}_{n,\nu}(\tau,z) 
\, = \, 
e^{\frac{\pi ina}{m}} F^{\, [m](a,b)}_{n,\nu}(\tau,z)
}}}
\\[2mm]
&=& 
e^{\frac{\pi i}{2m}(a^2+2na)} \, 
e^{\frac{\pi i}{2(m+\frac12)}(\nu+\frac12)^2} \, 
F^{\, [m](a,b)}_{n,\nu}(\tau,z)
\end{eqnarray*}}
proving \eqref{indef3:eqn:2023-1128a4}. 
\end{proof}

\medskip

\begin{note} 
\label{indef3:note:2023-1203a}
Let $m \in \frac12 \nnn$ and $n,n',\nu, \nu'$ be integers satisfying
the condition \eqref{indef3:eqn:2023-1127b1}. Then the following 
formula holds for $a \in \zzz/2m\zzz$. 
\begin{equation}
F^{\, [m](a,0)}_{n,\nu}(\tau,0) 
\,\ = \,\ 
(-1)^{2m} \, F^{\, [m](-a,0)}_{n',\nu'}(\tau,0)
\label{indef3:eqn:2023-1203a1}
\end{equation}
\end{note}

\begin{proof}  This formula follows from \eqref{indef3:eqn:2023-1125c} 
and Notes \ref{indef3:note:2023-1127a} and \ref{indef3:note:2023-1127b}.
\end{proof}

\medskip

From Proposition \ref{indef3:prop:2023-1128a}, we obtain the 
following:

\medskip

\begin{thm}
\label{indef3:thm:2023-1202a}
For $m \in \nnn$, the $\ccc$-linear span of 
$$
\Bigg\{F^{\, [m](a,0)}_{n, \nu}(\tau,0) \,\ ; \,\ 
\begin{array}{l}
n, \, \nu \in \zzz_{\geq 0} \quad \text{and} \quad 
a \in \zzz/2m\zzz \\[1mm]
\text{such that} \quad n < 2m, \,\ \nu \leq 2m
\end{array} \Bigg\}
$$
is $SL_2(\zzz)$-invariant.
\end{thm}

\medskip

In order to write the function $F^{\, [m](a,0)}_{n, \nu}(\tau,0)$ 
explicitly, we note the following:

\medskip


\begin{note} 
\label{indef3:note:2023-1130a}
Let $m \in \frac12 \nnn$ and $n$ and $\nu$ be integers 
satisfying \eqref{indef3:eqn:2023-1124a}. Then 
\begin{subequations}
\begin{enumerate}
\item[{\rm 1)}] \,\ $h^{[m]}_{n, \nu}(\tau, z) 
\,\ = \,\ 
e^{\pi i(n-2m\nu)z} \, q^{m\nu^2} \, 
\sum\limits_{j \in \zzz} \, 
\dfrac{e^{2\pi imjz} \, q^{mj^2+(n-2m\nu)j}}{1-e^{2\pi imz} \, q^{2mj}}$ 
\begin{equation}
= \,\ e^{\pi i(n-2m\nu)z} \, q^{m\nu^2} \, 
\Big[
\sum_{\substack{j, \, k \, \in \, \zzz \\[1mm] 
j \, \geq \, k \, \geq \, 0}}
-
\sum_{\substack{j, \, k \, \in \, \zzz \\[1mm] j \, < \, k \, < \, 0}}
\Big] \, 
e^{2\pi imjz} \, q^{m(j-\nu+\frac{n}{2m})^2 \, - \, 
m(k-\nu+\frac{n}{2m})^2}
\label{indef3:eqn:2023-1130a1}
\end{equation}
\item[{\rm 2)}] \,\ For ${}^{\forall}a \in \zzz$, the following formula holds:
\begin{equation}
h^{[m]}_{n, \nu}\Big(\tau, \, \frac{a\tau}{m}\Big)
\, = \, 
q^{m\nu^2-\frac{a^2}{4m}} \, 
\Big[
\sum_{\substack{j, \, k \, \in \, \zzz \\[1mm] 
j \, \geq \, k \, \geq \, 0}}
-
\sum_{\substack{j, \, k \, \in \, \zzz \\[1mm] j \, < \, k \, < \, 0}}
\Big] \, 
q^{m \, (j-\nu+\frac{n+a}{2m} )^2 \, - \, m(k-\nu+\frac{n}{2m})^2}
\label{indef3:eqn:2023-1130a2}
\end{equation}
\end{enumerate}
\end{subequations}
\end{note}

\begin{proof} These formulas can be checked by easy calculation.
\end{proof}

\medskip

Then we obtain an explicit formula for $F^{\, [m](a,0)}_{n, \nu}(\tau, \, 0)$
as follows:

\medskip


\begin{note} 
\label{indef3:note:2023-1130b}
Let $m \in \frac12 \nnn$ and $n$ and $\nu$ be integers 
satisfying \eqref{indef3:eqn:2023-1124a}. Then the following 
formulas hold for ${}^{\forall}a \in \zzz$:
\begin{subequations}
\begin{enumerate}
\item[{\rm 1)}] \,\ $G^{\, [m]}_{n, \nu}\Big(\tau, \, \dfrac{a\tau}{m}\Big)
\,\ = \,\ 
q^{\frac{1}{4m}[(n+a)^2-a^2]} \, 
\vartheta_{00}(2m\tau, \, (n+a)\tau) \, g^{[m]}_{n,\nu}(\tau)$
{\allowdisplaybreaks
\begin{eqnarray}
& & \hspace{-10mm}
+ \,\ (-1)^{\nu} \, 
q^{-\frac{a^2}{4m}} \, 
q^{\frac{1}{4(m+\frac12)}(\nu+\frac12)^2} \, 
\vartheta_{01}(2(m+\tfrac12)\tau, \, (\nu+\tfrac12)\tau)
\nonumber
\\[3mm]
& & 
\times \,\ \Big[
\sum_{\substack{j, \, k \, \in \, \zzz \\[1mm] 
j \, \geq \, k \, \geq \, 0}}
-
\sum_{\substack{j, \, k \, \in \, \zzz \\[1mm] j \, < \, k \, < \, 0}}
\Big] \, 
q^{m \, (j-\nu+\frac{n+a}{2m} )^2 \, - \, m(k-\nu+\frac{n}{2m})^2}
\label{indef3:eqn:2023-1130b1}
\end{eqnarray}}
\item[{\rm 2)}] \,\ $F^{\, [m](a,0)}_{n, \nu}(\tau, \, 0)
\,\ = \,\ 
q^{\frac{1}{4m}(n+a)^2} \, 
\vartheta_{00}(2m\tau, \, (n+a)\tau) \, g^{[m]}_{n,\nu}(\tau)$
{\allowdisplaybreaks
\begin{eqnarray}
& & \hspace{-10mm}
+ \,\ (-1)^{\nu} \, 
q^{\frac{1}{4(m+\frac12)}(\nu+\frac12)^2} \, 
\vartheta_{01}(2(m+\tfrac12)\tau, \, (\nu+\tfrac12)\tau)
\nonumber
\\[3mm]
& & 
\times \,\ \Big[
\sum_{\substack{j, \, k \, \in \, \zzz \\[1mm] 
j \, \geq \, k \, \geq \, 0}}
-
\sum_{\substack{j, \, k \, \in \, \zzz \\[1mm] j \, < \, k \, < \, 0}}
\Big] \, 
q^{m \, (j-\nu+\frac{n+a}{2m} )^2 \, - \, m(k-\nu+\frac{n}{2m})^2}
\label{indef3:eqn:2023-1130b2}
\end{eqnarray}}
Note that the Mumford's theta functions 
$\vartheta_{00}(2m\tau, \, (n+a)\tau)$ and 
$\vartheta_{01}(2(m+\tfrac12)\tau, \, (\nu+\tfrac12)\tau)$
have product expression.
\end{enumerate}
\end{subequations}
\end{note}

\begin{proof} 1) \,\ By \eqref{indef3:eqn:2023-1124b3} and 
\eqref{indef3:eqn:2023-1130a2}, one has
{\allowdisplaybreaks
\begin{eqnarray*}
& & \hspace{-15mm}
G^{\, [m]}_{n,\nu}\Big(\tau, \frac{a\tau}{m}\Big) 
\, = \, 
g^{[m]}_{n,\nu}(\tau) \, \theta_{n,m}\Big(\tau, \frac{a\tau}{m}\Big)
\, + \,
(-1)^{\nu} \, 
q^{-m\nu^2} h^{[m]}_{n,\nu}\Big(\tau, \frac{a\tau}{m}\Big) \, 
\theta_{\nu+\frac12, m+\frac12}^{(-)}(\tau,0)
\\[3mm]
&=&
g^{[m]}_{n,\nu}(\tau) \, \times \, 
q^{\frac{1}{4m}[(n+a)^2-a^2]} \, \vartheta_{00}(2m\tau, \, (n+a)\tau)
\\[2mm]
& &
+ \,\ (-1)^{\nu} \, 
q^{-\frac{a^2}{4m}} \, 
\Big[
\sum_{\substack{j, \, k \, \in \, \zzz \\[1mm] 
j \, \geq \, k \, \geq \, 0}}
-
\sum_{\substack{j, \, k \, \in \, \zzz \\[1mm] j \, < \, k \, < \, 0}}
\Big] \, 
q^{m \, (j-\nu+\frac{n+a}{2m} )^2 \, - \, m(k-\nu+\frac{n}{2m})^2}
\\[2mm]
& & \hspace{10mm}
\times \,\ 
q^{\frac{1}{4(m+\frac12)}(\nu+\frac12)^2} \, 
\vartheta_{01}(2(m+\tfrac12)\tau, \, (\nu+\tfrac12)\tau)
\end{eqnarray*}}
since
\begin{equation}
\left\{
\begin{array}{lcl}
\theta_{n,m}(\tau, \frac{a\tau}{m}) 
&=& 
q^{\frac{1}{4m}[(n+a)^2-a^2]} \, \vartheta_{00}(2m\tau, \, (n+a)\tau)
\\[3mm]
\theta_{\nu+\frac12, m+\frac12}^{(-)}(\tau,0) 
&=&
q^{\frac{1}{4(m+\frac12)}(\nu+\frac12)^2} \, 
\vartheta_{01}(2(m+\tfrac12)\tau, \, (\nu+\tfrac12)\tau)
\end{array}\right. 
\label{indef3:eqn:2023-1130b3}
\end{equation}
proving \eqref{indef3:eqn:2023-1130b1}. \,\ 
2) follows from 1) and \eqref{indef3:eqn:2023-1125c}.
\end{proof}

\section{The case $\nu=m$}


In this section, we consider the case when $m \in \nnn$ and 
$\nu=m$. In this case, since \, 
$\theta_{m+\frac12,m+\frac12}^{(-)}(\tau,0)=0$, 
the functions $G^{\, [m]}_{n,m}(\tau,z)$ and 
$F^{\, [m](a,b)}_{n,m}(\tau,z)$ are given as follows:
\begin{subequations}
{\allowdisplaybreaks
\begin{eqnarray}
G^{\, [m]}_{n,m}(\tau,z) &=& g^{[m]}_{n,m}(\tau) \, 
\theta_{n,m}(\tau,z)
\label{indef3:eqn:2023-1128b1}
\\[2mm]
F^{\, [m](a,b)}_{n,m}(\tau,z) &=& 
e^{\pi iaz} \, 
q^{\frac{a^2}{4m}} \, 
g^{[m]}_{n,m}(\tau) \, \theta_{n,m}
\Big(\tau, \, z+\frac{a\tau}{m}+\frac{b}{m}\Big)
\label{indef3:eqn:2023-1128b2}
\end{eqnarray}}
\end{subequations}
Then the function 
\begin{equation}
g^{[m]}_{n,m}(\tau) \, = \, 
\Big[\sum_{\substack{j, \, p \, \in \, \zzz \\[1mm]
0 \, < \, p \, \leq \, j}} 
-
\sum_{\substack{j, \, p \, \in \, \zzz \\[1mm]
j \, < \, p \, \leq \, 0}} \Big] 
(-1)^j \,
q^{(m+\frac12)(j+m-\frac12)^2 \, - \, m \, (p+m-\frac{n}{2m})^2} 
\label{indef3:eqn:2023-1203b1}
\end{equation}
is written, by \eqref{indef3:eqn:2023-1128b1} 
and Proposition \ref{indef3:prop:2023-1124a}, as follows:
\begin{subequations}
\begin{equation}
g^{[m]}_{n,m}(\tau)
\, = \,\ 
- \, \frac{(-1)^m}{2m} \cdot 
\frac{\eta(\tau)^3}{\theta_{n,m}(\tau,z)} 
\sum_{\ell \, \in \zzz/2m\zzz} 
e^{\frac{\pi in\ell}{m}} \,\ 
\dfrac{\theta_{m+\frac12, m+\frac12}^{(-)}\Big(\tau, \, z-\dfrac{\ell}{m}\Big)
}{
\theta_{\frac12, \frac12}^{(-)}\Big(\tau, \, z-\dfrac{\ell}{m}\Big)}
\label{indef3:eqn:2023-1203b2}
\end{equation}
The RHS of this equation \eqref{indef3:eqn:2023-1203b2} should be 
independent of $z$, since the LHS is. So, letting $z=0$ in 
\eqref{indef3:eqn:2023-1203b2}, we obtain the following expression 
for $g^{[m]}_{n,m}(\tau)$:
\begin{equation}
g^{[m]}_{n,m}(\tau)
\, = \,\ 
- \, \frac{(-1)^m}{2m} \cdot 
\frac{\eta(\tau)^3}{\theta_{n,m}(\tau,0)} 
\sum_{\ell \, \in \zzz/2m\zzz} 
e^{\frac{\pi in\ell}{m}} \,\ 
\dfrac{\theta_{m+\frac12, m+\frac12}^{(-)}\Big(\tau, \, -\dfrac{\ell}{m}\Big)
}{
\theta_{\frac12, \frac12}^{(-)}\Big(\tau, \, -\dfrac{\ell}{m}\Big)}
\label{indef3:eqn:2023-1203b3}
\end{equation}
\end{subequations}

\medskip

\begin{lemma} 
\label{indef3:lemma:2023-1129a}
For $m \in \nnn$ and $n, a, b \in \zzz$ such that $0 \leq n < 2m$, the 
following formula holds:
{\allowdisplaybreaks
\begin{eqnarray}
& &
F^{\, [m](a,b)}_{n,m}
\Big(-\dfrac{1}{\tau}, \dfrac{z}{\tau}\Big) 
\nonumber
\\[2mm]
& & \hspace{-5mm}
= \,\ 
\dfrac{(-i\tau)^{\frac12}}{\sqrt{2m}}
e^{\frac{\pi inb}{m}} \, 
e^{\frac{\pi iaz}{\tau}} \, 
e^{-\frac{\pi ia^2}{2m\tau}} \, 
e^{\frac{\pi im}{2\tau} (z-\frac{a}{m})^2}
g^{[m]}_{n, \nu}\Big(-\frac{1}{\tau}\Big) \, 
\sum_{k \in \zzz/2m\zzz}
e^{-\frac{\pi i(n+a)k}{m}} \, \theta_{k,m}(\tau, z)
\label{indef3:eqn:2023-1129a1}
\end{eqnarray}}
\end{lemma}

\begin{proof} We compute the $S$-transformation of 
$F^{\, [m](a,b)}_{n,m}(\tau,z)$ by using the formula 
\eqref{indef3:eqn:2023-1128b2} 
and Lemmas 1.2 and 1.3 in \cite{W2022c} as follows: 
{\allowdisplaybreaks
\begin{eqnarray*}
& & \hspace{-10mm}
F^{\, [m](a,b)}_{n,m}\Big(-\frac{1}{\tau}, \frac{z}{\tau}\Big) 
= \, 
e^{\frac{\pi iaz}{\tau}} \, 
e^{-\frac{2\pi i}{\tau} \cdot \frac{1}{4m}a^2} \, 
g^{[m]}_{n, m}\Big(-\frac{1}{\tau}\Big) \, 
\theta_{n,m}\Big(-\frac{1}{\tau}, 
\frac{z}{\tau}-\frac{a}{m\tau}+\frac{b}{m} 
\Big)
\\[2mm]
&=&
\frac{(-i\tau)^{\frac12}}{\sqrt{2m}}
e^{\frac{\pi iaz}{\tau}} \, 
e^{-\frac{\pi ia^2}{2m\tau}} \, 
e^{\frac{\pi im}{2\tau} (z-\frac{a}{m})^2} \, 
e^{\pi ib(z-\frac{a}{m})} \, 
e^{\frac{\pi ib^2\tau}{2m}} \, 
g^{[m]}_{n, m}\Big(-\frac{1}{\tau}\Big) \, 
\\[2mm]
& & \hspace{10mm}
\times \,\ 
\sum_{k \in \zzz/2m\zzz}
e^{-\frac{\pi ink}{m}} \, 
\theta_{k,m}\Big(\tau, z-\frac{a}{m}+\frac{b\tau}{m}\Big)
\\[2mm]
&=&
\frac{(-i\tau)^{\frac12}}{\sqrt{2m}}
e^{\frac{\pi iaz}{\tau}} \, 
e^{-\frac{\pi ia^2}{2m\tau}} \, 
e^{\frac{\pi im}{2\tau} (z-\frac{a}{m})^2} \, 
g^{[m]}_{n, m}\Big(-\frac{1}{\tau}\Big) \, 
\sum_{k \in \zzz/2m\zzz}
e^{-\frac{\pi ink}{m}} \, e^{-\frac{\pi ia(k+b)}{m}} \, 
\theta_{k+b,m}(\tau, z)
\end{eqnarray*}}
Then, putting $k+b=:k'$, we obtain the formula \eqref{indef3:eqn:2023-1129a1}.
\end{proof}

\medskip

\begin{lemma}
\label{indef3:lemma:2023-1129b}
For $m \in \nnn$ and $n \in \zzz$ such that $0 \leq n < 2m$, the 
following formula holds:
\begin{equation}
{ } \hspace{-10mm}
g^{[m]}_{n, m}
\Big(-\dfrac{1}{\tau}\Big) \, \theta_{0,m}(\tau, z)
\, = \, 
\dfrac{(-1)^{m-1} \, \tau}{\sqrt{2m(2m+1)}} 
\sum\limits_{n'=0}^{2m-1}
\sum\limits_{\nu'=0}^{2m}
e^{\frac{\pi inn'}{m}} \, F^{\, [m](-n', 0)}_{n',\nu'}(\tau,z)
\label{indef3:eqn:2023-1129a2}
\end{equation}
\end{lemma}

\begin{proof} Letting $\nu=m$ in \eqref{indef3:eqn:2023-1128a1}, 
we have
{\allowdisplaybreaks
\begin{eqnarray}
& & \hspace{-10mm}
F^{\, [m](a,b)}_{n,m}\Big(-\frac{1}{\tau}, \frac{z}{\tau}\Big)
= \,\ 
- \, i \, (-1)^m \, 
\frac{(-i\tau)^{\frac32}}{2m\sqrt{2m+1}} \, 
e^{\frac{\pi iaz}{\tau}} \, 
e^{-\frac{\pi ia^2}{2m\tau}} \, 
e^{\frac{\pi im}{2\tau}(z-\frac{a}{m})^2} 
\sum_{n'=0}^{2m-1}
\sum_{\nu'=0}^{2m}
\sum_{\ell \, \in \, \zzz/2m\zzz} 
\nonumber
\\[0mm]
& & \hspace{30mm}
\times \,\ 
e^{\frac{\pi in(b-\ell)}{m}} \, 
e^{-\frac{\pi i\ell a}{m}} \, 
F^{\, [m](\ell, -a)}_{n',\nu'}(\tau,z)
\label{indef3:eqn:2023-1129b1}
\end{eqnarray}}
The, by \eqref{indef3:eqn:2023-1129a1} and 
\eqref{indef3:eqn:2023-1129b1} and Note 
\ref{indef3:note:2023-1128a}, we have
{\allowdisplaybreaks
\begin{eqnarray}
& &
g^{[m]}_{n, m}\Big(-\frac{1}{\tau}\Big) \, 
\sum_{k \in \zzz/2m\zzz}
e^{-\frac{\pi i(n+a)k}{m}} \, \theta_{k,m}(\tau, z)
\nonumber
\\[2mm]
&=&
(-1)^{m-1} \, 
\frac{\tau}{\sqrt{2m(2m+1)}} \,\ 
\sum_{n'=0}^{2m-1}
\sum_{\nu'=0}^{2m} \, 
\sum_{\ell \, \in \, \zzz/2m\zzz} 
e^{-\frac{\pi in\ell}{m}} \, 
e^{-\frac{\pi i\ell a}{m}} \, 
F^{\, [m](\ell, -a)}_{n',\nu'}(\tau,z)
\nonumber
\\[2mm]
&=&
(-1)^{m-1} \, 
\frac{\tau}{\sqrt{2m(2m+1)}} \,\ 
\sum_{n'=0}^{2m-1}
\sum_{\nu'=0}^{2m} \, 
\sum_{\ell \, \in \, \zzz/2m\zzz} 
e^{-\frac{\pi in\ell}{m}} \, 
e^{-\frac{\pi i(\ell+n') a}{m}} \, 
F^{\, [m](\ell, 0)}_{n',\nu'}(\tau,z)
\label{indef3:eqn:2023-1129b2}
\end{eqnarray}}
Then, applying \, $\sum_{a \in \zzz/2m\zzz}$ to 
\eqref{indef3:eqn:2023-1129b2}, we obtain
$$
g^{[m]}_{n, m}\Big(-\frac{1}{\tau}\Big) \, \theta_{0,m}(\tau, z)
\, = \, 
(-1)^{m-1} \, \frac{\tau}{\sqrt{2m(2m+1)}} 
\sum_{n'=0}^{2m-1}
\sum_{\nu'=0}^{2m}
e^{\frac{\pi inn'}{m}} \, F^{\, [m](-n', 0)}_{n',\nu'}(\tau,z)
$$
proving \eqref{indef3:eqn:2023-1129a2}.
\end{proof}

\medskip

We note that the function $F^{\, [m](-n,0)}_{n,\nu}(\tau,0)$ 
has the following expression:

\medskip

\begin{note}  
\label{indef3:note:2023-1202a}
Let $m \in \frac12 \nnn$ and $n,n',\nu, \nu'$ be integers satisfying
the condition \eqref{indef3:eqn:2023-1127b1}. Then
{\allowdisplaybreaks
\begin{eqnarray}
& & \hspace{-7mm}
F^{\, [m](-n,0)}_{n,\nu}(\tau,0) 
\,\ = \,\ 
\dfrac{\eta(2m\tau)^5}{\eta(m\tau)^2\eta(4m\tau)^2} \, 
g^{[m]}_{n,\nu}(\tau)
\nonumber
\\[2mm]
& & \hspace{-5mm}
+ \, (-1)^{\nu}  
q^{\frac{1}{4(m+\frac12)}(\nu+\frac12)^2} 
\vartheta_{01}(2(m+\tfrac12)\tau, (\nu+\tfrac12)\tau) 
\Big[
\sum_{\substack{j, \, k \, \in \, \zzz \\[1mm] 
j \, \geq \, k \, \geq \, 0}}
-
\sum_{\substack{j, \, k \, \in \, \zzz \\[1mm] j \, < \, k \, < \, 0}}
\Big] \, 
q^{m (j-\nu)^2 - m(k-\nu+\frac{n}{2m})^2}
\nonumber
\\[-7mm]
& &
\label{indef3:eqn:2023-1202a1}
\end{eqnarray}}
\end{note}

\begin{proof} This formula follows immediately from 
\eqref{indef3:eqn:2023-1130b2} and 
$\vartheta_{00}(2m\tau,0) \, = \, 
\dfrac{\eta(2m\tau)^5}{\eta(m\tau)^2\eta(4m\tau)^2}
$ \, .
\end{proof}

\medskip


\begin{prop}  
\label{indef3:prop:2023-1202a}
For $m \in \nnn$ and $n \in \zzz$ such that $0 \leq n < 2m$, the 
$S$-transformation of the function $g^{[m]}_{n, m}(\tau)$ is given 
by the following formula:
\begin{subequations}
{\allowdisplaybreaks
\begin{eqnarray}
& & \hspace{-15mm}
g^{[m]}_{n, m}\Big(-\frac{1}{\tau}\Big)
\, = \, 
\frac{(-1)^{m-1} \, \tau}{\sqrt{2m(2m+1)}} \cdot 
\frac{\eta(m\tau)^2\eta(4m\tau)^2}{\eta(2m\tau)^5}
\sum_{n'=0}^{2m-1}
\sum_{\nu'=0}^{2m}
e^{\frac{\pi inn'}{m}} \, F^{\, [m](-n', 0)}_{n',\nu'}(\tau,0)
\label{indef3:eqn:2023-1202b1}
\\[2mm]
& & \hspace{-6.5mm}
= \,\ - \,\ 
\frac{\tau}{\sqrt{2m(2m+1)}} \, \Bigg\{
(-1)^m 
\sum_{n'=0}^{2m-1}
\sum_{\nu'=0}^{2m}
e^{\frac{\pi inn'}{m}} \, g^{[m]}_{n',\nu'}(\tau)
\nonumber
\\[2mm]
& & 
+ \,\ \frac{\eta(m\tau)^2\eta(4m\tau)^2}{\eta(2m\tau)^5}  
\sum_{n'=0}^{2m-1}
\sum_{\nu'=0}^{2m}
(-1)^{m+\nu'} \, e^{\frac{\pi inn'}{m}} \, 
q^{\frac{1}{4(m+\frac12)}(\nu'+\frac12)^2} 
\nonumber
\\[2mm]
& & \hspace{5mm}
\times \,\ 
\vartheta_{01}(2(m+\tfrac12)\tau, (\nu'+\tfrac12)\tau) \, 
\Big[
\sum_{\substack{j, \, k \, \in \, \zzz \\[1mm] 
j \, \geq \, k \, \geq \, 0}}
-
\sum_{\substack{j, \, k \, \in \, \zzz \\[1mm] j \, < \, k \, < \, 0}}
\Big] \, 
q^{m (j-\nu')^2 - m(k-\nu'+\frac{n'}{2m})^2} \Bigg\}
\label{indef3:eqn:2023-1202b2}
\end{eqnarray}}
\end{subequations}
\end{prop}

\begin{proof} The formula \eqref{indef3:eqn:2023-1202b1} is obtained 
immediately by letting $z=0$ in \eqref{indef3:eqn:2023-1129a2} and 
using the Gauss' formula in Exercise 12.4 in \cite{K1} : \quad $
\theta_{0,m}(\tau,0) \, = \, 
\dfrac{\eta(2m\tau)^5}{\eta(m\tau)^2\eta(4m\tau)^2}$ \, .

\noindent
The formula \eqref{indef3:eqn:2023-1202b2} follows from 
\eqref{indef3:eqn:2023-1202b1} and \eqref{indef3:eqn:2023-1202a1}.
\end{proof}


\begin{thebibliography}{99}

\bibitem{K1} V. G. Kac : Infinite-Dimensional Lie Algebras, 3rd edition,
Cambridge University Press, 1990.

\bibitem{KP} V. G. Kac and D. Peterson : Infinite-dimensional
Lie algebras, theta functions and modular forms, 
Advances in Math. 53 (1984), 125-264.




\bibitem{KW2014} V. G. Kac and M. Wakimoto : 
Representations of affine superalgebras and mock theta functions, 
Transformation Groups 19 (2014), 387-455.
arXiv:1308.1261.

\bibitem{KW2016a} V. G. Kac and M. Wakimoto :
Representations of affine superalgebras and mock theta functions II, 
Advances in Math. 300 (2016), 17-70. 
arXiv:1402.0727.

\bibitem{KW2016b} V. G. Kac and M. Wakimoto :
Representations of affine superalgebras and mock theta functions III, 
Izv. Math. 80 (2016), 693-750. 
arXiv:1505.01047.

\bibitem{KW2017b} V. G. Kac and M. Wakimoto : Representation of superconformal 
algebras and mock theta functions, Trudy Moskow Math. Soc. 78 (2017), 64-88.
arXiv:1701.03344.

\bibitem{Mum} D. Mumford : Tata Lectures on Theta I, Progress in Math. 28, 
Birkh\"{a}user Boston, 1983.



\bibitem{W2022a} M. Wakimoto : Mock theta functions and characters of 
N=3 superconformal modules, arXiv:2202.03098.

\bibitem{W2022b} M. Wakimoto : Mock theta functions and characters of 
N=3 superconformal modules II, arXiv:2204.01473.

\bibitem{W2022c} M. Wakimoto : Mock theta functions and 
indefinite modular forms, arXiv:2206.02293.

\bibitem{W2022d} M. Wakimoto : Mock theta functions and characters of 
N=3 superconformal modules III, arXiv:2207.04644.

\bibitem{W2022e} M. Wakimoto : Mock theta functions and characters of 
N=3 superconformal modules IV, arXiv:2209.00234.

\bibitem{W2022f} M. Wakimoto : Mock theta functions and 
indefinite modular forms II, arXiv:2210.02280.


\bibitem{W2023b} M. Wakimoto : A note on Appell's functions related to 
the denominators of affine Lie superalgebras $\widehat{sl}(2|1)$ and 
$\widehat{osp}(3|2)$, arXiv:2305.08618.

\bibitem{Z} S. Zwegers : Mock theta functions, PhD Thesis, Universiteit 
Utrecht, 2002, arXiv:0807.483.


\end{thebibliography}
\end{document}